\documentclass[reqno]{amsart}

\usepackage{amssymb,amsthm,amsmath} 
\usepackage[top=25truemm,bottom=25truemm,left=25truemm,right=25truemm]{geometry}
\usepackage{graphicx}
\usepackage{latexsym}
\usepackage{comment}
\usepackage{setspace}
\usepackage{mathrsfs}

\numberwithin{equation}{section}
\theoremstyle{plain}
\newtheorem{theorem}{Theorem}[section]
\newtheorem{lemma}[theorem]{Lemma}
\newtheorem{proposition}[theorem]{Proposition}
\newtheorem{corollary}[theorem]{Corollary}

\theoremstyle{definition}
\newtheorem{definition}[theorem]{Definition}
\newtheorem{remark}[theorem]{Remark}

\newcommand{\ad}{\,\mathrm{ad}} 

\begin{document}

\title[Inverse scattering for the nonlocal nonlinear Schr\"odinger equation]
{On inverse scattering for the nonlocal nonlinear Schr\"odinger equation with slowly decaying data}

\author[Hideo Takaoka]{Hideo Takaoka}

\address{Hideo Takaoka \newline
\indent Department of Mathematics, Kobe University, Kobe, 657-8501, Japan}
\email{takaoka@math.kobe-u.ac.jp}


\begin{abstract}
This paper discusses about an inverse scattering problem to the nonlocal nonlinear Schr\"odinger equation.
We extend the previous result treated by Y. Zhao and E. Fan of inverse scattering result at the point with slowly decaying data.
More precisely, we show the global existence of solution in $H^{1,s}$ for $s>1/2$ by assuming small data in $L^1$.
The proof relies on the inverse scattering method based on the ZS-AKNS (Zakharov-Shabat, Ablowitz-Kaup-Newell-Segur) eigenvalue problem on the spectral theory of ordinary differential equations upon the framework of slowly decaying data.
\end{abstract}

\maketitle

\section{Introduction}\label{sec:intro}
We consider the global existence of solutions to the initial value problem for the following nonlocal nonlinear Schr\"odinger equation
\begin{equation}\label{eq:nnls}
\begin{gathered}
i\partial_t q(t,x)+\partial_x^2 q(t,x)=2q(t,x)r(t,x)q(t,x),\quad (t,x)\in\mathbb{R}^2,\\
q(0,x)=q_0(x),\quad x\in\mathbb{R},
\end{gathered}
\end{equation}
where $r(t,x)=\sigma \overline{q(t,-x)}$ and $\sigma\in\{+1,-1\}$.
This equation was first introduced by Ablowitz and Musslimani \cite{ablowitz2} as a new nonlocal reduction of the Ablowitz-Kaup-Newell-Segur (AKNS) system \cite{ablowitz1} and be known to be completely integrable.
It is invariant under the combined action of parity operation ($\mathcal{P}:x\to -x$) and time-reversal operator ($\mathcal{T}:i\to -i$), hence self-induced potential $V(t,x)= 2q(t,x)r(t,x)$ is $\mathcal{PT}$ symmetric, $V(t,x)=\overline{V(t,-x)}$.
Such $\mathcal{PT}$ symmetric arises in optical systems, see \cite{regensburger,ruter}.

The nonlocal nonlinear Schr\"odinger equation \eqref{eq:nnls} is the Hamiltonian equation with respect to the following canonical Poisson structure:
$$
\{F,G\}=\frac{1}{i}\int_{\mathbb{R}} \left(\frac{\delta F}{\delta q(x)}\frac{\delta G}{\delta r(x)}-\frac{\delta F}{\delta r(x)}\frac{\delta G}{\delta q(x)}\right)\,dx.
$$
In particular, the equation \eqref{eq:nnls} is the Hamiltonian flow associated to
\begin{equation}\label{eq:hamiltonian}
H=\int_{\mathbb{R}} \left(q_x(t,x)r_x(t,x)+q^2(t,x)r^2(t,x)\right)\,dx,
\end{equation}
namely
$$
\partial_t q(x)=\{q(x),H\}=\frac{1}{i}\int_{\mathbb{R}}\frac{\delta H}{\delta r(y)}\delta(x-y)\,dy=\frac{1}{i}\frac{\delta H}{\delta r(x)}.
$$
The equation \eqref{eq:nnls} possesses the Lax operators:
\begin{equation}\label{eq:lax1}
L(q)=
\begin{bmatrix}
\kappa-\partial_x & q \\
-r & \kappa+\partial_x \\
\end{bmatrix},
\qquad
P(q)=i
\begin{bmatrix}
2\partial_x^2-qr & -q_x-2q\partial_x \\
r_x+2r\partial_x & -2\partial_x^2+qr \\
\end{bmatrix},
\end{equation}
where $\kappa\in\mathbb{C}$ denotes the spectral parameter.
Then the equation \eqref{eq:nnls} admits a Lax pair representation $\partial_tL=[P,L]$.

There is an extensive literature on the subject to the nonlinear Schr\"odinger equation such as replacing $r(t,x)=\sigma \overline{q(t,-x)}$ with $r(t,x)=\sigma \overline{q(t,x)}$ in \eqref{eq:nnls}, that is,
\begin{equation}\label{eq:nls}
 i\partial_t q(t,x)+\partial_x^2 q(t,x)=2\sigma|q(t,x)|^2q(t,x).
\end{equation}
The equation \eqref{eq:nls} was proven to be a completely integrable by mathematicians Zakharov and Shabat \cite{zakharov}.
They designed the inverse scattering transform to solve that initial value problem of the equation \eqref{eq:nls} for initial data of which decays rapidly as $x\to\pm\infty$.
By the inverse scattering transform, the equation is completely integrable establishing that it possesses an infinite number of conservation laws.
Ablowitz, Kaup, Newell and Segur \cite{ablowitz1} developed a method to solve an initial value problem of a broad class of nonlinear evolution equations by various reductions of the spectral problem.
For the defocusing case $\sigma=1$, Deift and Zhou \cite{deift} obtained the long-time asymptotic for solutions in the setting $H^{1,1}$.
As to the existing well-posedness result, Tsutsumi \cite{tsutsumi} proved global well-posedness of the corresponding initial value problem for the equation \eqref{eq:nls} in $L^2$ (see also \cite{cazenave}). 
More recently, Harrop-Griffiths, Killip and Vi\c{s}an \cite{harrop} extended the global well-posedness with low regularity in $H^s$ for $s>-1/2$, which is almost optimal according to the scaling analysis:
$$
q_{\lambda}(t,x)=\lambda q(\lambda^2 t,\lambda x).
$$

On the other hands, only few results are known on the equation \eqref{eq:nnls}.
The local well-posedness in $L^2$ holds in a similar way to the method in \cite{cazenave,tsutsumi} with some minor changes, but we know less for global existence of solutions due to the lack of a positive-definite and conserved energy of Hamiltonian \eqref{eq:hamiltonian}.
When $\sigma=-1$, both equations of \eqref{eq:nnls} and \eqref{eq:nls} have the same global in time soliton solutions
\begin{equation}\label{eq:soliton}
e^{\omega^2t}Q_{\omega}(x)\quad (\omega>0)
\end{equation}
where $Q_{\omega}(x)=\omega/(e^{\omega x}+e^{-{\omega x}})$ is the ground state solutions of \eqref{eq:nls}.
The ground state solution $Q_{\omega}(x)$ refers to the eigenfunction corresponding to the first eigenvalue (the smallest eigenvalue) of the Hamiltonian operator for the equation \eqref{eq:nls}.
Genoud \cite{genoud} proved that the soliton solutions \eqref{eq:soliton} to \eqref{eq:nnls} are unstable by blowing up singularities near the origin (see e.g., \cite{gurses,rybalko}). 
In \cite{takaoka}, the author proved that the solution remains close to the orbit of the soliton for a long time, if the initial data is close to the ground state soliton. 

As classical integrable equations, the inverse scattering method addresses large-time behavior of solutions to \eqref{eq:nnls} by proving global well-posedness and determining large-time asymptotic behavior for soliton-free initial data.
Recently, the existence of global solutions in the weighted Sobolev space $H^{1,1}$ with the smallness assumption on the $L^1$ norm was studied by Zhao and Fan \cite{zhao}, based upon the inverse scattering theory (see e.g., \cite{ablowitz3}).
To perform inverse scattering, it is required that a potential must be small in $L^1$,  that is, the total spatial variation of scatterers is small in $L^1$.
In view of the fact that
$$
\|Q_{\omega}\|_{L^1}=\|Q_1\|_{L^1}=\int_{\mathbb{R}}\frac{dx}{e^x+e^{-x}}=2\int_1^{\infty}\frac{dt}{t^2+1}=\frac{\pi}{2},
$$
the inverse scattering method cannot capture soliton solutions in \eqref{eq:soliton}.

Let us turn now to our inverse scattering results in $H^{1,s}$.
The aim is to push down the decaying rate $s$ in weighed Lebesgue space $L^{2,s}$ to $s>1/2$.
This type of result was also studied by Cuccagnaa and Pelinovsky \cite{cuccagna} for the equation \eqref{eq:nls}.

\begin{definition}
Let $1\le p\le \infty$, denote by $L^p(\mathbb{R})$ the standard Lebesgue space.
We set $L^p=L^p(\mathbb{R})$ for simplicity. 
For $s \in\mathbb{R}$, we denote the weighted Lebesgue space $L^{p,s}$ of functions $f$ such that $\langle x\rangle^sf(x)\in L^p(\mathbb{R})$, where $\langle x\rangle=(1+|x|^2)^{1/2}$.
Moreover, denote by $H^s$ the usual Sobolev space for $s\in\mathbb{R}$.
For $s\in\mathbb{R}$, we denote the weighted Sobolev space $H^{1,s}$, whose norm is
$$
\|f\|_{H^{1,s}}=\|f\|_{H^1}+\|f\|_{L^{2,s}}.
$$
Let
$X^s$ be spaces equipped with the norm
$$
\|f\|_{X^s}=\|f\|_{L^{2,1}}+\|f\|_{H^s}.
$$
\end{definition}

Our aim is to prove the following result described below.

\begin{theorem}\label{thm:main}
Let $s>1/2$.
There exists $\varepsilon_0>0$ such that for any initial data $q_0\in H^{1,s}$ with $\|q_0\|_{L^1}\le \varepsilon_0$, there exists a unique global solution $q(t)\in C_t(\mathbb{R},H_x^{1,s})$ to the initial value problem \eqref{eq:nnls}.
Moreover, the map $q_0\mapsto q$ is Lipschitz continuous from $H^{1,s}$ to $C_t(\mathbb{R},H_x^{1,s})$.
\end{theorem}

\begin{remark} 
We comment on how our work is related to the inverse scattering problem for the equation \eqref{eq:nls}.
An inverse scattering transform for the defocusing nonlinear Schr\"odinger equation ($\sigma=1$ in \eqref{eq:nls}) has be carried out for arbitrarily large data in $L^1$ (see e.g., \cite{ablowitz1,deift}).
Without major revision, the proof in this manuscript can also be applied to the defocusing nonlinear Schr\"odinger equation \eqref{eq:nls} under no small data hypothesis (see Remark \ref{rem:nosmall} in the paper and also e.g., \cite{cuccagna}).
For the focussing case ($\sigma=-1$ in \eqref{eq:nls}), we require the small data hypothesis, which avoids soliton solutions.
\end{remark}

The inverse scattering problems when $q\in H^s\cap L^{2,s}$ is open.
Making an inverse scattering problem in this function space is a challenging topic to research.

As the remainder of the paper, we will constrain $0<s<1$.
Throughout the paper, we restrict our attention and consider to the case when $0<s<1$.
 
In closing the introduction section, we indicate the organization of the paper.
In Section \ref{sec:notations}, we describe some basic notation.
In Section \ref{sec:akns}, we characterize the scattering data of the ZS-AKNS system associated with the Lax pair.
In Section \ref{sec:ist}, the inverse scattering transform is considered, which is a nonlinear analogue of the Fourier transform.
In Section \ref{sec:rhproblem}, we formulate the associated Riemann-Hilbert problem to determine generalized Jost solutions of ZS-AKNS system.
In Section \ref{sec:rp}, the scattering solution, the potential $q$ is reconstructed from the Riemann-Hilbert problem.
In Section \ref{sec:main}, we then show Theorem \ref{thm:main}. 

\section{Notation}\label{sec:notations}

In this section, we introduce notation and collect several useful results to be utilized throughout the paper.
Let $e_1$ and $e_2$ be unit vectors in $\mathbb{R}^2$: 
$$
e_1=
\begin{bmatrix}
1 \\
0
\end{bmatrix}, \quad
e_2= 
\begin{bmatrix}
0 \\
1
\end{bmatrix}.
$$
For a matrix $A(x)=[a_{ij}(x)]$, let
$$
|A(x)|=\left(\sum_{ij}|a_{ij}(x)|^2\right)^{1/2}=(\mathrm{tr} A(x)^*A(x))^{1/2}
$$
be the Frobenius norm of $A(x)$ and letting $\|A\|_{L^p}=\||A(\cdot)|\|_{L^p_x}$.

Let $\sigma_3$ be the third Pauli matrix (like a reflection about the $x$ axis):
$$
\sigma_3=
\begin{bmatrix}
1 & 0 \\
0 & -1 \\
\end{bmatrix}.
$$

The notation $\mathbb{R}_{+}$ refers to the set of non-negative real numbers, while $\mathbb{R}_-$ represents $\mathbb{R}_-=-\mathbb{R}_+$.
Also, we use $\mathbb{C}_+=\{z\in\mathbb{C}\mid \Im z>0\}$ and $\mathbb{C}_-=\{z\in\mathbb{C}\mid \Im z<0\}$.

To distinguish the coordinate in multi-component spaces, we sometimes use the expression $\mathbb{R}_{x}$ to denote the set of real numbers assigned to a specific coordinate axis, like the $x$-axis.
Moreover, the function space $L^2_x(\mathbb{R})$ is square-integrable functions defined over the real numbers of $x$.

If $s>1/2$, we have $L^{2,s}(\mathbb{R})\hookrightarrow L^1(\mathbb{R})$ and $H^s(\mathbb{R})\hookrightarrow C(\mathbb{R})\cap L^{\infty}(\mathbb{R})$.

If $0<s<1$, then we use Plancherel's formula in $d$-dimensional Euclidean space $\mathbb{R}^{d}$ to obtain the following \cite{bergh}:
$$
\int_{\mathbb{R}^d\times \mathbb{R}^d}\frac{|f(x)-f(y)|^2}{|x-y|^{d+2s}}\,dxdy=C(d,s)\int_{\mathbb{R}^d}|\xi|^{2s}|\mathscr{F}[f](\xi)|^2\,d\xi,
$$
where 
$$
C(d,s)=2\int_{\mathbb{R}^d}\frac{1-\cos t_1}{|t|^{d+2s}}\,dt,
$$
and $\mathscr{F}[f]$ is the Fourier transform of $f$ as
$$
\mathscr{F}[f](\xi)=\frac{1}{(2\pi)^{d/2}}\int_{\mathbb{R}^n}e^{-i\xi\cdot x}f(x)\,dx.
$$

The Cauchy operator on $\mathbb{R}$ is given by
$$
\mathscr{C}[h](z)=\frac{1}{2\pi i}\int_{\mathbb{R}}\frac{h(s)}{s-z}\,ds,\quad z\in\mathbb{C}\backslash \mathbb{R}
$$
for suitable functions $f$ on $\mathbb{R}$.
Moreover, let
$$
\mathscr{C}^{\pm}[h](k)=\lim_{\varepsilon\downarrow 0}\mathscr{C}[h](k\pm i \varepsilon),\quad k\in\mathbb{R}
$$
be the Plemelj-Sokhotski operators over $\mathbb{R}$.
Noting that by the Sokhotski-Plemelj formula 
$$
\mathscr{C}^{\pm}[f]=\pm \frac12f+\frac{i}{2}\mathscr{H}f
$$
for a Schwartz function $f$, where $\mathscr{H}$ is the Hilbert transform operator as
$$
\mathscr{H}[f](x)=\frac{1}{\pi}\lim_{\varepsilon\downarrow 0}\int_{|s-x|>\varepsilon}\frac{f(s)}{x-s}\,ds.
$$
Then 
\begin{equation}\label{eq:C-iden}
\mathscr{C}^{+}[f]-\mathscr{C}^{-}[f]=f,\quad \mathscr{C}^{+}[f]+\mathscr{C}^{-}[f]=i\mathscr{H}[f],
\end{equation}
on the Schwartz function space.
Since the Hilbert transform $\mathscr{H}$ is bounded operator from $L^p$ to itself for $1<p<\infty$, the Cauchy operator and the Plemelj-Sokhotski operator do so.
The Plemelj-Sokhotski operators are orthogonal projectors on $L^2$, namely
\begin{equation}\label{eq:C-orth}
\mathscr{C}^{+}\mathscr{C}^{+}=\mathscr{C}^+,\quad \mathscr{C}^-\mathscr{C}^-=\mathscr{C}^-,\quad \mathscr{C}^+\mathscr{C}^-=\mathscr{C}^-\mathscr{C}^+=0.
\end{equation}
Moreover
$$
\mathscr{C}\left[\mathscr{C}^{\pm}[h]\right](z)=
\begin{cases}
\pm \mathscr{C}[h](z), & z\in\mathbb{C}_{\pm},\\
0, & z\in \mathbb{C}_{\mp},
\end{cases}
$$
respectively, on the Schwartz function space.

\section{ZS-AKNS scattering problem}\label{sec:akns}

Throughout the section, we omit the time variable $t$ and will write $q(t,x)$ like $q(x)$ according to conventions.
By replacing $\kappa=ik$ in \eqref{eq:lax1}, we introduce he ZS-AKNS operator
\begin{equation*}
T=-\sigma_3L(q)=\partial_xI-ik\sigma_3-Q(q),
\quad
Q(q)=
\begin{bmatrix}
0 & 	q \\
r & 0 \\
\end{bmatrix}.
\end{equation*}
Let us consider the eigen-solutions $\psi=\psi(x,k)$ associated with Lax operators:
\begin{equation*}
T\psi=0,
\end{equation*}
where $\psi$  is a square matrix of two dimensions.
By setting $m=\psi(x,k) e^{-ixk\sigma_3}$, the equation $T\psi=0$ is transformed into
\begin{equation}\label{eq:eq-m}
\partial_xm=ik\ad\sigma_3(m)+Q(q)m,
\end{equation}
where $\ad\,\sigma_3(m)=[\sigma_3,m]=\sigma_3 m-m\sigma_3$ is the adjoint action of $\sigma_3$ on $m$. 
Denote by $m^{(\pm)}(x,k)$ the solution of \eqref{eq:eq-m} with the asymptotic condition at $\pm\infty$
\begin{equation*}\label{eq:boundary}
\lim_{x\to \pm\infty}m^{(\pm)}(x,k)=I,
\end{equation*}
respectively.
By Hadamard's lemma $e^{\ad A}B=e^ABe^{-A}$, we have the Jost matrix $m^{(\pm)}$ based on the Volttera integral equations:
\begin{equation}\label{eq:volttera}
m^{(\pm)}(x,k)  =I+K_{q,k,\pm}[m^{(\pm)}(\cdot,k)](x)
\end{equation}
where
$$
K_{q,k,\pm}[\psi](x)=\int_{\pm\infty}^xe^{ik(x-y)\ad\sigma_3}Q(q(y))\psi(y)\,dy.
$$

Let $m_1^{(\pm)}$ and $m_2^{(\pm)}$ be the first and second columns of $m^{(\pm)}$, namely
$$
m^{(\pm)}(x,k)=
\begin{bmatrix}
m_1^{(\pm)}(x,k) & m_2^{(\pm)}(x,k)
\end{bmatrix}.
$$
By indicating the entries of a vector by subscripts, we write
$$
m_1^{(\pm)}(x,k)=
\begin{bmatrix}
m_{1,1}^{(\pm)}(x,k)\\
m_{1,2}^{(\pm)}(x,k)
\end{bmatrix},
\quad
m_2^{(\pm)}(x,k)=
\begin{bmatrix}
m_{2,1}^{(\pm)}(x,k)\\
m_{2,2}^{(\pm)}(x,k)
\end{bmatrix}
$$
When using these notations, we have from \eqref{eq:volttera} that
\begin{equation}\label{eq:m_1}
\begin{split}
m_1^{(\pm)}(x,k)& =e_1
+\int_{\pm\infty}^x
\begin{bmatrix}
1 & 0 \\
0 & e^{-2ik(x-y)}  
\end{bmatrix}
Q(q(y))m_1^{(\pm)}(y,k)\,dy\\
& =
\begin{bmatrix}
1\\
0
\end{bmatrix}
+\int_{\pm\infty}^x
\begin{bmatrix}
q(y)m_{1,2}^{(\pm)}(y,k) \\
e^{-2ik(x-y)}\sigma\overline{q(-y)}m_{1,1}^{(\pm)}(y,k)
\end{bmatrix}
\,dy,
\end{split}
\end{equation}
\begin{equation}\label{eq:m_2}
\begin{split}
m_2^{(\pm)}(x,k) & =e_2
+\int_{\pm\infty}^x
\begin{bmatrix}
e^{2ik(x-y)} & 0 \\
0 & 1 
\end{bmatrix}
Q(q(y))m_2^{(\pm)}(y,k)\,dy\\
&  =
\begin{bmatrix}
0\\
1
\end{bmatrix}
+\int_{\pm\infty}^x
\begin{bmatrix}
e^{2ik(x-y)}q(y)m_{2,2}^{(\pm)}(y,k) \\
\sigma\overline{q(-y)}m_{2,1}^{(\pm)}(y,k)
\end{bmatrix}
\,dy,
\end{split}
\end{equation}
which reduces from \eqref{eq:m_2} that
\begin{equation*}
\begin{bmatrix}
\overline{m_{2,2}^{(\mp)}(-x,k)}\\
-\sigma\overline{m_{2,1}^{(\mp)}(-x,k)}
\end{bmatrix}
=
\begin{bmatrix}
1\\
0
\end{bmatrix}
+\int_{\pm\infty}^x
\begin{bmatrix}
q(y)(-\sigma\overline{m_{2,1}^{(\mp)}(-y,k)})\\
e^{2i\overline{k}(x-y)}\sigma\overline{q(-y)}\overline{m_{2,2}^{(\mp)}(-y,k)}
\end{bmatrix}
\,dy
\end{equation*}
Then we expect that 
\begin{equation}\label{eq:equality}
\begin{bmatrix}
m_{1,1}^{(\pm)}(x,k)\\
m_{1,2}^{(\pm)}(x,k)
\end{bmatrix}
=
\begin{bmatrix}
\overline{m_{2,2}^{(\mp)}(-x,-\overline{k})}\\
-\sigma\overline{m_{2,1}^{(\mp)}(-x,-\overline{k})}
\end{bmatrix}
=
\begin{bmatrix}
0 & 1 \\
-\sigma & 0
\end{bmatrix}
\overline{
\begin{bmatrix}
m_{2,1}^{(\mp)}(-x,-\overline{k})\\
m_{2,2}^{(\mp)}(-x,-\overline{k})
\end{bmatrix}}.
\end{equation}

The existence of solutions for \eqref{eq:volttera} can be proven by iteration method.
We recall the following result, see \cite[pp.~1050]{deift} (see also \cite[Lemma 2,1]{zhao}).

\begin{lemma}\label{lem:m^+-}
Suppose $q\in L^1$.
Then for all $k\in\mathbb{R}$,  the solutions $m^{(\pm)}$ of the Volterra integral equations \eqref{eq:volttera} exist uniquely in $C_x(\mathbb{R})\cap L_x^{\infty}(\mathbb{R})$.
Moreover, for all $x\in\mathbb{R}$, $m_1^{(+)}(x,k)$ in \eqref{eq:m_1} and $m_2^{(-)}(x,k)$ in \eqref{eq:m_2} are analytic for $k\in\mathbb{C}_+$, continuous for $k\in\overline{\mathbb{C}_+}$, while $m_1^{(-)}(x,k)$ in \eqref{eq:m_1} and $m_2^{(+)}(x,k)$ in \eqref{eq:m_2} are analytic for $k\in\mathbb{C}_-$, continuous for $k\in\overline{\mathbb{C}_-}$.
\end{lemma}

\begin{proof}
In the proof, we sketch the proof of existence of solution $m^{(-)}(x,k)$ in \eqref{eq:m_1}  without giving all the details, since this lemma is straightforward adaptations of  \cite[Lemma 2,1]{zhao}.

One may solve the equation $m_1^{(-)}(x,k)$ by using iteration.
$$
m_{(n+1)}^{(-)}(x,k)=I+K_{q,k,-}[m_{(n)}^{(-)}(\cdot,k)](x).
$$
We make an inductive hypothesis:
\begin{equation*}
m_{(n)}^{(-)}(x,k)=\sum_{\ell=0}^nI_{\ell}(x,k).
\end{equation*}
A simple computation yields that $I_l(x,k)$ is the $\ell$-fold integral satisfying
$$
|I_{\ell}(x,k)|\le \frac{1}{\ell !}\left(\int_{-\infty}^x|Q(q(y)|\,dy\right)^{\ell}.
$$
Then we have that the infinite series
$$
m^{(-)}(x,k)=\sum_{\ell=1}^{\infty}I_{\ell}(x,k)
$$
converges absolutely, and
$$
|m^{(-)}(x,k)|\le e^{\int_{-\infty}^x|Q(q(y))|\,dy}=e^{\sqrt{2}\int_{-\infty}^x|q(y)|\,dy},
$$
which ultimately yields the existence of solution $m^{(-)}(x,k)$.
\end{proof}

We expand out the limit in \eqref{eq:boundary} and will have the following lemma.
The following lemma is due to see \cite[Lemma 2,2]{zhao} (see also \cite[pp.~1030]{deift}).

\begin{lemma}\label{lem:m-exist}
Let $m^{(\pm)}$ be solutions of \eqref{eq:volttera} obtained in Lemma $\ref{lem:m^+-}$ under the condition that $q\in L^1$.
For all $x\in\mathbb{R}$, we have
\begin{equation}\label{eq:m-limit1}
\lim_{\Im k\to\pm\infty}m_1^{(\pm)}(x,k)=e_1,\quad \lim_{\Im k\to\pm\infty}m_2^{(\mp)}(x,k)=e_2,
\end{equation}
respectively.
Moreover, if $q\in H^s$ with $s>1/2$, then for all $x\in\mathbb{R}$, we have
\begin{equation}\label{eq:m-limit2}
\lim_{\Im k\to\pm\infty}2ik
\begin{bmatrix}
m_1^{(\pm)}(x,k)-e_1 & m_2^{(\mp)}(x,k)-e_2
\end{bmatrix}
=
\begin{bmatrix}
\int_{\pm\infty}^xq(y)r(y)\,dy  & -q(x) \\
r(x) & -\int_{\pm\infty}^xq(y)r(y)\,dy
\end{bmatrix},
\end{equation}
respectively.
\end{lemma}

\begin{proof}
This  lemma is due to see \cite[Lemma 2,2]{zhao} (see also \cite[pp.~1030]{deift}).
See also \eqref{eq:I-K} below.
\end{proof}

We consider the regularity of $m_1^{(\pm)}(x,k)$ and $m_2^{(\pm)}(x,k)$ with respect to the variable $k$.

\begin{lemma}\label{lem:H^s}
Let $q\in L^{2,s}$ with $s>1/2$.
For all $x\in\mathbb{R}_{\pm}$, we have
\begin{equation}\label{eq:H^s-regularity}
m^{(\pm)}(x,k)-I \in H_k^s
\end{equation} 
and
\begin{equation}\label{eq:L^s-weight}
m^{(\pm)}(x,k)-I\in L^{\infty,s}_x(\mathbb{R}_{\pm},L^2_k(\mathbb{R})),
\end{equation}  
respectively.
Moreover, if $q\in H^{1,s}$ with $s>1/2$, then for all $x\in\mathbb{R}$ we have
\begin{equation}\label{eq:L_k^2}
2ik
\begin{bmatrix}
m_1^{(\pm)}(x,k)-e_1 & m_2^{(\mp)}(x,k)-e_2
\end{bmatrix}
-\begin{bmatrix}
\int_{\pm\infty}^xq(y)r(y)\,dy  & -q(x) \\
r(x) & -\int_{\pm\infty}^xq(y)r(y)\,dy
\end{bmatrix}\in L_k^2.
\end{equation}
\end{lemma}

\begin{proof}
First, we consider the proof of \eqref{eq:H^s-regularity} and \eqref{eq:L^s-weight} in term of $m_1^{\pm}$.
By \eqref{eq:volttera}, we have
\begin{equation*}
\begin{split}
(I-K_{q,k,\pm})[m_1^{(\pm)}(\cdot,k)-e_1](x)= & e_1-(I-K_{q,k,\pm})[e_1](x)\\
= & K_{q,k,\pm}[e_1](x)\\
= & \int_{\pm\infty}^xe^{-2ik(x-y)}r(y)\,dy \,e_2\\
= & \mp(2\pi)^{1/2}\mathscr{F}[r(x+\cdot)1_{\mathbb{R}_{\pm}}](-2k)\,e_2.
\end{split}
\end{equation*}
By virtue of \cite[Theorem 3.2]{deift}, we have
\begin{equation}\label{eq:k^bound}
\|(I-K_{q,k,\pm})^{-1}\|_{L_x^{\infty}L_k^2\to L_x^{\infty}L_k^2}\le e^{\|Q(q)\|}=e^{\sqrt{2}\|q\|_{L^1}}
\end{equation}
so that for $x\in \mathbb{R}_{\pm}$
\begin{equation}\label{eq:m_1L^2}
\begin{split}
\|m_1^{(\pm)}(x,k)-e_1\|_{L_k^2}&  \lesssim e^{\sqrt{2}\|q\|_{L^1}}\|r(x+\cdot)1_{\mathbb{R}_{\pm}}\|_{L^2_y}\\
& \lesssim e^{\sqrt{2}\|q\|_{L^1}}\langle x\rangle^{-s}\|q\|_{L_x^{2,s}},
\end{split}
\end{equation}
which proves $m_1^{\pm}(x,\cdot)-e_1\in L_k^2$ and \eqref{eq:L^s-weight} in the term of $m_1^{(\pm)}$.

Moreover, it follows from \eqref{eq:m_1} that
\begin{equation*}
\begin{split}
(I-K_{q,k+h,\pm})[m_1^{(\pm)}(\cdot,k+h)-m_1^{(\pm)}(\cdot,k)](x) = & (K_{q,h+k,\pm}-K_{q,h,\pm})[m_1^{(\pm)}(\cdot,k)-e_1](x)\\
& +\int_{\pm\infty}^x\left(e^{-2i(k+h)(x-y)}-e^{-2ik(x-y)}\right)r(y)\,dy\,e_2.
\end{split}
\end{equation*}
In a similar fashion to using \eqref{eq:k^bound}, we have 
\begin{equation}\label{eq:m_1-11}
\begin{split}
\|m_1^{(\pm)}(x,\cdot+h)-m_1^{(\pm)}(x,\cdot)\|_{L_k^2} \le  & e^{\sqrt{2}\|q\|_{L^1}}\|(K_{q,h+k,\pm}-K_{q,h,\pm})[m_1^{(\pm)}(\cdot,k)-e_1](x)\|_{L_k^2}\\
& +e^{\sqrt{2}\|q\|_{L^1}}\left\|\int_{\pm\infty}^x\left(e^{-2i(k+h)(x-y)}-e^{-2ik(x-y)}\right)r(y)\,dy\right\|_{L_k^2}.
\end{split}
\end{equation}
For the first term in \eqref{eq:m_1-11}, we have
\begin{equation*}
\begin{split}
(K_{q,h+k,\pm}-K_{q,h,\pm})[m_1^{(\pm)}(\cdot,k)-e_1](x)= & -(2\pi)^{1/2}\mathscr{F}[r(x+\cdot)m_{1,2}^{(\pm)}(x+\cdot,k)1_{\mathbb{R}_{\pm}}](-k-h) \,e_2\\
& +(2\pi)^{1/2} \mathscr{F}[r(x+\cdot)m_{1,2}^{(\pm)}(x+\cdot,k)1_{\mathbb{R}_{\pm}}](-k)\, e_2\\
= & (2\pi)^{1/2}ih\int_0^1\mathscr{F}_y[yr(x+y)m_{1,2}^{(\pm)}(x+y,k)1_{\mathbb{R}_{\pm}}(y)](-k-h\theta)\,d\theta\,e_2,
\end{split}
\end{equation*}
and thence
\begin{equation}\label{eq:m_1-1}
\begin{split}
|(K_{q,h+k,\pm}-K_{q,h,\pm})[m_1^{(\pm)}(\cdot,k)-e_1](x)|\lesssim   |h|\left|\int_{\mathbb{R}_{\pm}}|yr(x+y)m_{1,2}^{(\pm)}(x+y,k)|\,dy\right|.
\end{split}
\end{equation}
According to another calculation, we deduce that
\begin{equation}\label{eq:m_1-2}
\begin{split}
|(K_{q,h+k,\pm}-K_{q,h,\pm})[m_1^{(\pm)}(\cdot,k)-e_1](x)|\le   2\left|\int_{\mathbb{R}_{\pm}}|r(x+y)m_{1,2}^{(\pm)}(x+y,k)|\,dy\right|.
\end{split}
\end{equation}
Interpolating these estimates in \eqref{eq:m_1-1} and \eqref{eq:m_1-2}, we have
\begin{equation*}
\begin{split}
|(K_{q,h+k,\pm}-K_{q,h,\pm})[m_1^{(\pm)}(\cdot,k)-e_1](x)|\lesssim \left|\int_{\mathbb{R}_{\pm}}\min\{1,|hy|\}|r(x+y)m_{1,2}^{(\pm)}(x+y,k)|\,dy\right|.
\end{split}
\end{equation*}
We then take the $L_k^2(\mathbb{R})$ norm and obtain the following estimate for $x\in\mathbb{R}_{\pm}$;
\begin{equation}\label{eq:m_1-first}
\begin{split}
& \int_{\mathbb{R}}\frac{\|(K_{q,h+k,\pm}-K_{q,h,\pm})[m_1^{(\pm)}(\cdot,k)-e_1](x)\|_{L_k^2}^2}{|h|^{1+2s}}\,dh\\
&\lesssim  \int_{\mathbb{R}}\frac{1}{|h|^{1+2s}}\left(\int_{\mathbb{R}_{\pm}}\min\{1,|hy|\}|r(x+y)|\|m_1^{(\pm)}(x+y,\cdot)-e_1\|_{L^2_k}\,dy\right)^2\,dh\\
&\lesssim \int_{\mathbb{R}}\frac{1}{|h|^{1+2s}}\left(\int_{\mathbb{R}_{\pm}}\frac{\min\{1,|hy|\}}{\langle x+y\rangle^s}|r(x+y)|\,dy\right)^2\,dh\sup_{x\in \mathbb{R}_{\pm}}\langle x\rangle^{2s}\|m_1{(x,k)}-e_1\|_{L^2_k}^2\\
& \lesssim \left(\int_{\mathbb{R}_{\pm}}\left(\int_{\mathbb{R}}\frac{\min\{1,|hy|^2\}}{|h|^{1+2s}}\,dh\right)^{1/2}\frac{|r(x+y)|}{\langle x+y\rangle^s}\,dy\right)^2\sup_{x\in \mathbb{R}_{\pm}}\langle x\rangle^{2s}\|m_1{(x,k)}-e_1\|_{L^2_k}^2\\
& \lesssim \left(\int_{\mathbb{R}_{\pm}}|r(x+y)|\,dy\right)^2\sup_{x\in \mathbb{R}_{\pm}}\langle x\rangle^{2s}\|m_1{(x,k)}-e_1\|_{L^2_k}^2\\
& \lesssim \left(\|r\|_{L^{2,s}(\mathbb{R}_{\pm})}\sup_{x\in \mathbb{R}_{\pm}}\langle x\rangle^{s}\|m_1{(x,k)}-e_1\|_{L^2_k}\right)^2.
\end{split}
\end{equation}
To estimate the second term in \eqref{eq:m_1-11}, we observe that
$$
\int_{\pm\infty}^x\left(e^{-2i(k+h)(x-y)}-e^{-2ik(x-y)}\right)r(y)=-(2\pi)^{1/2}\mathscr{F}[r(x+\cdot)1_{\mathbb{R}_{\pm}}](-2(k+h))+(2\pi)^{1/2}\mathscr{F}[r(x+\cdot)1_{\mathbb{R}_{\pm}}](-2k).
$$
Then we get for $x\in\mathbb{R}_{\pm}$;
\begin{equation*}
\begin{split} \int_{\mathbb{R}}\frac{1}{|h|^{1+2s}}\left\|\int_{\pm\infty}^x\left(e^{-2i(k+h)(x-y)}-e^{-2ik(x-y)}\right)r(y)\right\|_{L^2_k}^2\,dh
& \lesssim  \| |y|^s r(x+y)1_{\mathbb{R}_{\pm}}(y)\|_{L^2_y}^2\\
& \lesssim \|r\|_{L^{2,s}(\mathbb{R}_{\pm})}^2.
\end{split}
\end{equation*}
Together with \eqref{eq:m_1L^2}  and \eqref{eq:m_1-first}, we obtain \eqref{eq:H^s-regularity} for the term $m_1^{(\pm)}(x,k)-I$.
The proof of estimates \eqref{eq:H^s-regularity} and \eqref{eq:L^s-weight} for the term of $m_2^{(\pm)}$ follows analogously. 

We next consider the estimates in \eqref{eq:L_k^2}.
Since $q\in H^{1,s}$ guarantees that for all $\xi\in\mathbb{R}$;
$$
2ik\int_{\pm\infty}^xe^{-2ik(x-y)}r(y)\,dy-r(x)=-\int_{\pm\infty}^xe^{-2ik(x-y)}r_x(y)\,dy
$$
(c.f. \cite[Corollary 8.10]{brezis}), we use the integration by parts on any intervals to obtain that
\begin{equation}\label{eq:I-K}
\begin{split}
& (I-K_{q,k,\pm})\left[ k(m_1^{(\pm)}(\cdot,k)-e_1)-\frac{1}{2i}\int_{\pm\infty}^{\cdot}q(y)r(y)\,dy\,e_1-\frac{1}{2i}r(\cdot)\,e_2\right](x)\\
& = \frac{1}{2i}\left(2ik\int_{\pm\infty}^xe^{-2ik(x-y)}r(y)\,dy-r(x)+\int_{\pm\infty}^xe^{-2ik(x-y)}r(y)\int_{\pm\infty}^yq(t)r(t)\,dt\right)\,e_2\\
& = \frac{1}{2i}\left(-\int_{\pm\infty}^xe^{-2ik(x-y)}r_x(y)\,dy+\int_{\pm\infty}^xe^{-2ik(x-y)}r(y)\int_{\pm\infty}^yq(t)r(t)\,dt\right)\,e_2.
\end{split}
\end{equation}
Recall by Plancherel's theorem, it follows that
$$
\sup_{x\in\mathbb{R}}\left\|-\int_{\pm\infty}^xe^{-2ik(x-y)}r_x(y)\,dy+\int_{\pm\infty}^xe^{-2ik(x-y)}r(y)\int_{\pm\infty}^yq(t)r(t)\,dt\right\|_{L^2_k}\lesssim \|r_x\|_{L^2}+\|r\|_{L^2}^3.
$$
Applying \eqref{eq:k^bound} to the above derives, we obtain \eqref{eq:L_k^2} for the term $m_1^{(\pm)}(x,k)-I$.
The proof of the estimate for the term $m_2^{(\pm)}(x,k)-I$ follows analogously. 
\end{proof}

In view of the proof of Lemma \ref{lem:H^s}, we modify the proof to show the following.

\begin{proposition}\label{prop:q-m-lipschitz}
Suppose $q\in L^{2,s}$ with $s>1/2$.
Then the mapping
$$
\begin{array}{ccc}
L^{2,s}                     &\longrightarrow& L^{\infty}_x(\mathbb{R}_{\pm},H^s_k(\mathbb{R}))\cap L_x^{\infty,s}(\mathbb{R}_{\pm}, L_k^2(\mathbb{R}))                    \\
\rotatebox{90}{$\in$}&               & \rotatebox{90}{$\in$} \\
 q                    & \longmapsto   & m^{(\pm)}-I
\end{array}
$$
is locally Lipschitz.
Moreover, if $q\in H^{1,s}(\mathbb{R})$, then the mapping
$$
\begin{array}{ccc}
H^{1,s}                     &\longrightarrow& L^{\infty}_x(\mathbb{R},L^2_k(\mathbb{R}))                    \\
\rotatebox{90}{$\in$}&               & \rotatebox{90}{$\in$} \\
 q                    & \longmapsto   & \widetilde{m^{(\pm)}}
\end{array}
$$
is locally Lipschitz, where
$$
\widetilde{m^{(\pm)}}
=
2ik
\begin{bmatrix}
m_1^{(\pm)}(x,k)-e_1 & m_2^{(\mp)}(x,k)-e_2
\end{bmatrix}
-
\begin{bmatrix}
\int_{\pm\infty}^xq(y)r(y)\,dy  & -q(x) \\
r(x) & -\int_{\pm\infty}^xq(y)r(y)\,dy
\end{bmatrix}.
$$
\end{proposition}

\section{Scattering problem}\label{sec:ist}

In this section, we investigate the inverse scattering transform for the equation \eqref{eq:nnls}.
We introduce the reflection coefficients, which are a nonlinear analogue of the Fourier transform of $q(t,x)$.

By the equation \eqref{eq:eq-m}, one sees the following:
\begin{equation}\label{eq:a-d}
\begin{split}
& \partial_x\det \begin{bmatrix} m_1^{(+)}(x,k) & m_2^{(-)}(x,k)\end{bmatrix}=0\\
& \partial_x\det \begin{bmatrix} m_1^{(-)}(x,k) & m_1^{(+)}(x,k)\end{bmatrix}=-2ik\det \begin{bmatrix} m_1^{(-)}(x,k) & m_1^{(+)}(x,k)\end{bmatrix},\\
& \partial_x\det \begin{bmatrix} m_2^{(+)}(x,k) & m_2^{(-)}(x,k)\end{bmatrix}=2ik\det \begin{bmatrix} m_2^{(+)}(x,k) & m_2^{(-)}(x,k)\end{bmatrix},\\
& \partial_x\det \begin{bmatrix} m_1^{(-)}(x,k) & m_2^{(+)}(x,k)\end{bmatrix}=0.
\end{split}
\end{equation}
Let $S(x,k)$ be a scattering matrix associated with the equation \eqref{eq:eq-m}:
$$
S(x,k)=
\begin{bmatrix}
a(k) & c(k)e^{2ikx} \\
b(k) e^{-2ikx} & d(k)
\end{bmatrix}
=e^{ikx \ad\sigma_3}
\begin{bmatrix}
a(k) & c(k) \\
b(k) & d(k)
\end{bmatrix},
$$
which being a continuous matrix function given by
\begin{equation}\label{eq:m-m}
m^{(+)}(x,k)=m^{(-)}(x,k)S(x,k).
\end{equation}
By Cramer's rule and \eqref{eq:a-d}, we find
\begin{equation*}
\begin{split} 
a(k) & =\det \begin{bmatrix} m_1^{(+)}(0,k) & m_2^{(-)}(0,k)\end{bmatrix},\\
b(k) & =\det \begin{bmatrix} m_1^{(-)}(0,k) & m_1^{(+)}(0,k)\end{bmatrix},\\
c(k) & =\det \begin{bmatrix} m_2^{(+)}(0,k) & m_2^{(-)}(0,k)\end{bmatrix},\\
d(k) & =\det \begin{bmatrix} m_1^{(-)}(0,k) & m_2^{(+)}(0,k)\end{bmatrix}.
\end{split}
\end{equation*}
From \eqref{eq:m-m}, we compute
\begin{equation}\label{eq:a-d=1}
a(k)d(k)-b(k)c(k)=1.
\end{equation}
On the other hand, application of \eqref{eq:m_1} and \eqref{eq:m_2} yields the following analogous formulas:
\begin{equation}\label{eq:a(k)}
\begin{split}
a(k) & 
=
\begin{bmatrix}
m_1^{(+)}(0,k)-e_1 & m_2^{(-)}(0,k)-e_2
\end{bmatrix}
+
\begin{bmatrix}
e_1 & m_2^{(-)}(0,k)-e_2
\end{bmatrix}
+
\begin{bmatrix}
m_1^{(+)}(0,k)-e_1 & e_2 
\end{bmatrix}
+1\\
& = 
\begin{cases}
1-\int_{\mathbb{R}}q(y)m_{1,2}^{(+)}(y,k)\,dy,\\
1+\int_{\mathbb{R}}r(y)m_{2,1}^{(-)}(y,k)\,dy,\\
m_{1,1}^{(+)}(-\infty,k),\\
m_{2,2}^{(-)}(+\infty,k).
\end{cases}
\end{split}
\end{equation}
\begin{equation}\label{eq:b(k)}
\begin{split}
b(k) & = 
\det\begin{bmatrix}
m_1^{(-)}(0,k)-e_1 & m_1^{(+)}(0,k)-e_1
\end{bmatrix}
+
\det\begin{bmatrix}
e_1 & m_1^{(+)}(0,k)-e_1
\end{bmatrix}
+\det
\begin{bmatrix}
m_1^{(-)}(0,k)-e_1 & e_1
\end{bmatrix}\\
& = 
\begin{cases}
-\int_{\mathbb{R}}e^{2iky}r(y)m_{1,1}^{(+)}(y,k)\,dy,\\
-\int_{\mathbb{R}}e^{2iky}r(y)m_{1,1}^{(-)}(y,k)\,dy,
\end{cases}
\end{split}
\end{equation}
\begin{equation}\label{eq:c(k)}
\begin{split}
c(k) & = 
\det\begin{bmatrix}
m_2^{(+)}(0,k)-e_2 & m_2^{(-)}(0,k)-e_2
\end{bmatrix}
+
\det\begin{bmatrix}
e_2 & m_2^{(-)}(0,k)-e_2
\end{bmatrix}
+\det
\begin{bmatrix}
m_2^{(+)}(0,k)-e_2 & e_2
\end{bmatrix}\\
& =
\begin{cases}
-\int_{\mathbb{R}}e^{-2iky}q(y)m_{2,2}^{(+)}(y,k)\,dy,\\
-\int_{\mathbb{R}}e^{-2iky}q(y)m_{2,2}^{(-)}(y,k)\,dy,
\end{cases}
\end{split}
\end{equation}
\begin{equation}\label{eq:d(k)}
\begin{split}
d(k) & =
\det\begin{bmatrix}
m_1^{(-)}(0,k)-e_1 & m_2^{(+)}(0,k)-e_2
\end{bmatrix}
+
\det\begin{bmatrix}
e_1 & m_2^{(+)}(0,k)-e_2
\end{bmatrix}
+\det
\begin{bmatrix}
m_1^{(-)}(0,k)-e_1 & e_2
\end{bmatrix}
+1\\
& =
\begin{cases}
1+\int_{\mathbb{R}}q(y)m_{1,2}^{(-)}(y,k)\,dy,\\
1-\int_{\mathbb{R}}r(y)m_{2,1}^{(+)}(y,k)\,dy,\\
m_{11}^{(-)}(+\infty,k),\\
m_{22}^{(+)}(-\infty,k).
\end{cases}
\end{split}
\end{equation}
Turing our attention to \eqref{eq:equality}, we notice that
\begin{equation}\label{eq:ad-relation}
a(k)  =\overline{a(-\overline{k})},\quad
d(k)  =\overline{d(-\overline{k})},\quad
b(k)  =\sigma \overline{c(-\overline{k})}.
\end{equation}

We consider the regularity of functions $a(k),~b(k),~c(k),~d(k)$.

\begin{lemma}\label{lem:a(k)-1}
Let $s>1/2$.
Suppose $q\in L^{2,s}$.
Then we have that
\begin{equation}\label{eq:a(k)-1}
S(0,\cdot) \in H^s_k.
\end{equation}
Moreover, if $q\in H^{1,s}$, then we have
\begin{equation}\label{eq:kb(k)}
b(k)\in L^{2,1}_k,
\end{equation}
that is, $b(k)\in X^s_k$.
\end{lemma}

\begin{proof}
We record the following by \eqref{eq:a(k)}, \eqref{eq:b(k)}, \eqref{eq:c(k)}, \eqref{eq:d(k)}:
\begin{equation*}
\begin{split}
& \|a(k)-1\|_{H^s_k}\lesssim \|m_1^{(+)}(0,\cdot)-e_1\|_{H^s_k}\|m_2^{(-)}(0,\cdot)-e_2\|_{H^s_k}+ \|m_2^{(-)}(0,\cdot)-e_2\|_{H^s_k}+ \|m_1^{(+)}(0,\cdot)-e_1\|_{H^s_k},\\
& \|b(k)\|_{H^s_k}\lesssim \|m_1^{(-)}(0,\cdot)-e_1\|_{H^s_k}\|m_1^{(+)}(0,\cdot)-e_1\|_{H^s_k}+ \|m_1^{(+)}(0,\cdot)-e_1\|_{H^s_k}+ \|m_1^{(-)}(0,\cdot)-e_1\|_{H^s_k},\\
& \|c(k)\|_{H^s_k}\lesssim \|m_2^{(-)}(0,\cdot)-e_2\|_{H^s_k}\|m_2^{(+)}(0,\cdot)-e_2\|_{H^s_k}+ \|m_2^{(+)}(0,\cdot)-e_2\|_{H^s_k}+ \|m_2^{(-)}(0,\cdot)-e_2\|_{H^s_k},\\
& \|d(k)-1\|_{H^s_k}\lesssim \|m_1^{(-)}(0,\cdot)-e_1\|_{H^s_k}\|m_2^{(+)}(0,\cdot)-e_2\|_{H^s_k}+ \|m_1^{(-)}(0,\cdot)-e_1\|_{H^s_k}+ \|m_2^{(+)}(0,\cdot)-e_2\|_{H^s_k},
\end{split}
\end{equation*}
where we use the fractional Leibniz rules and Sobolev inequality $H^s_k\hookrightarrow L^{\infty}_k$ for $s>1/2$.
Thus, \eqref{eq:a(k)-1} follows by \eqref{eq:H^s-regularity}.

We turn to \eqref{eq:kb(k)}.
Repeating the analysis used in \cite[Lemma 2.7]{zhao} and \cite[Theorem 3.3]{deift}, by \eqref{eq:b(k)} we find 
\begin{equation}\label{eq:b-reduction}
\begin{split}
kb(k)= & -k\int_{\mathbb{R}}e^{2iky}r(y)\,dy-\frac{1}{2i}\int_{\mathbb{R}}e^{2iky}r(y)\int_{+\infty}^yq(t)r(t)\,dtdy\\
& - \int_{\mathbb{R}}e^{2iky}r(y)\left(k(m_{1,1}^{(+)}(y,k)-1)-\frac{1}{2i}\int_{+\infty}^yq(t)r(t)\,dt\right)\,dy\\
 = & (2\pi)^{1/2}k\mathscr{F}[r](-2k)-\frac{(2\pi)^{1/2}}{2i}\mathscr{F}\left[ r(\cdot)\int_{+\infty}^{\cdot} q(t)r(t)\,dt\right](-2k)\\
& - \int_{\mathbb{R}}e^{2iky}r(y)\left(k(m_{1,1}^{(+)}(y,k)-1)-\frac{1}{2i}\int_{+\infty}^yq(t)r(t)\,dt\right)\,dy.
\end{split}
\end{equation}
To estimate the contribution of the third term on the right-hand side of \eqref{eq:b-reduction} to $\|kb(k)\|_{L^2_k}$, we bound
$$
\|q\|_{L^1}\left\|k(m_{1,1}^{(+)}(y,k)-1)-\frac{1}{2i}\int_{+\infty}^yq(t)r(t)\,dt\right\|_{L_y^{\infty}(\mathbb{R},L^{2}_k(\mathbb{R}))},
$$
which has an acceptable bound, since by Proposition \ref{prop:q-m-lipschitz}.
For the contributions of the first and second terms on the right-hand side of \eqref{eq:b-reduction} to $\|kb(k)\|_{L^2_k}$, we bound
$$
\|\partial_x r\|_{L^2}+\left\|r(x)\int_{+\infty}^{x} q(y)r(y)\,dy\right\|_{L^2_x},
$$
which has a bound
$$
\|q\|_{H^1}+\|q\|_{L^2}^{3}\lesssim \|q\|_{H^1}\left(1+\|q\|_{L^2}^2\right).
$$
This completes the proof of $b(k)\in L^{2,1}_k$.
\end{proof}

By Lemma \ref{lem:m^+-}, $a(k)$ and $d(k)$ are analytic in $\mathbb{C}_-$ and $\mathbb{C}_+$, respectively.
We will need to bring a lower bound of $|a(k)|$ and $|d(k)|$ strictly above zero. 

\begin{lemma}\label{lem:ad-lower-bound}
Let $s>1/2$.
Suppose $q\in L^{2,s}$ with $\|q\|_{L^1}\ll 1$.
Then there exists $c_0>0$ such that
$$
\inf_{k\in\mathbb{C}_-}|a(k)|\ge c_0,\quad
\inf_{k\in\mathbb{C}_+}|d(k)|\ge c_0.
$$
\end{lemma}

\begin{remark}\label{rem:nosmall}
Let us give a remark on the case of defocusing nonlinear Schr\"odinger equation ($\sigma=1$ in \eqref{eq:nls}).
For the case the equation \eqref{eq:nls}, the corresponding $a(k),b(k),c(k),d(k)$ satisfy
$$
a(k)=\overline{d(k)},\quad b(k)=\sigma\overline{c(k)}
$$
for $k\in\mathbb{R}$.
By \eqref{eq:a-d=1}, we have $|a(k)|^2-|b(k)|^2=1$, which leads us to exchange $|a(k)|\ge 1$ automatically, without small data hypothesis.
\end{remark}

\begin{proof}
Note that $q\in L^{2,s}\hookrightarrow L^1$.
$\|(I-K_{q,k,\pm})^{-1}\|_{ L_x^{\infty}\to L_x^{\infty}}\le e^{\|Q(q)\|_{L^1}}$ follows once for $k\in \mathbb{C}_{\pm}$.
With \eqref{eq:m_2}, we may bound
$$
\|m_2^{\pm}(x,k)-e_2\|_{L_k^{\infty}(\mathbb{C}_{\mp},L_x^{\infty}(\mathbb{R}))}\le e^{\|Q(q)\|_{L^1}}\int_{\mathbb{R}}|q(x)|\,dx,
$$
for $k\in \mathbb{C}_{\mp}$.
In view of \eqref{eq:a(k)}, for sufficient small $\|q\|_{L^1}$ we get that there exists a constant $c_0>0$ such that for all $k\in\mathbb{C}_-$  
$$
|a(k)|\ge 1 -\int_{\mathbb{R}}|r(x)||m_2^{(+)}(x,k)-e_2|\,dx\ge 1-e^{\|Q(q)\|_{L^1}}\int_{\mathbb{R}}|q(x)|\,dx\ge c_0.
$$
The estimate for $|d(k)|$ follows analogously. 
\end{proof}

With \eqref{eq:m-m} in mind, we observe that
\begin{equation}\label{eq:r-t0}
\begin{split}
\frac{1}{a(k)}m_1^{(+)}(x,k)-m_1^{(-)}(x,k) & =\frac{b(k)}{a(k)}e^{-2ikx}m_2^{(-)}(x,k),\\
\frac{1}{d(k)}m_2^{(+)}(x,k)-m_2^{(-)}(x,k) &=\frac{c(k)}{d(k)}e^{2ikx}m_1^{(-)}(x,k).
\end{split}
\end{equation}
Let $r(r_1(k),r_2(k))=r(r_1[q](k),r_2[q](k))$ be a pair of reflection coefficient induced by the potential $q$ such that
$$
r_1(k)=\frac{b(k)}{a(k)},\quad r_2(k)=\frac{c(k)}{d(k)}.
$$
Note that $1/a(k)$ and $1/d(k)$ are referred to as transmission coefficients.

\begin{lemma}\label{lem:r-upper-bound}
Let $s>1/2$.
Suppose $q\in L^{2,s}$ and $\|q\|_{L^1}\ll 1$.
Then we have
$$
\|r_1(k)\|_{L^{\infty}_k},~
\|r_2(k)\|_{L_k^{\infty}}<1.
$$
Moreover
$$
r_1(k),~r_2(k)\in H^s_k.
$$
\end{lemma}

\begin{proof}
The proof of $r_1(k)$ and $r_2(k)$ is both essentially identical, so that we consider the proof for the term $r_1(k)$.
Proceeding as in the proof of Lemma \ref{lem:ad-lower-bound}, we are led to obtain the following:
$$
 \|m_1^{(\pm)}(x,k)-e_1\|_{L_k^{\infty}(\mathbb{C}_{\pm},L^{\infty}_x(\mathbb{R}))}\le e^{\|Q(q)\|_{L^1}}\int_{\mathbb{R}}|q(x)|\,dx,
$$
which implies
$$
\|m_1^{(\pm)}(x,k)\|_{L_k^{\infty}(\mathbb{C}_{\pm},L^{\infty}_x(\mathbb{R}))}\le 1+e^{\|Q(q)\|_{L^1}}\|q\|_{L^1}.
$$
We substitute \eqref{eq:b(k)} then
$$
|r_1(k)|\le \frac{1}{c_0}\left(1+e^{\|Q(q)\|_{L^1}}\|q\|_{L^1}\right)\|q\|_{L^1}<1,
$$
provided if $\|q\|_{L^1}\ll 1$.

Next we seek to show $r_1(k)\in H_k^s$.
It suffices to bound
\begin{equation*}
\|r_1\|_{L_k^2}^2+\int_{\mathbb{R}}\frac{\|r_1(k+h)-r_1(k)\|_{L_k^2}^2}{|h|^{1+2s}}\,dh.
\end{equation*}
By Lemmas \ref{lem:a(k)-1} and \ref{lem:ad-lower-bound}, the contribution of the first term is easily seen to be acceptable:
$$
 \|r_1\|_{L_k^2}^2\le \frac{\|b(k)\|_{L^2_k}^2}{c_0^2}.
$$
To estimate the contribution of the second term, we use again Lemmas \ref{lem:a(k)-1} and \ref{lem:ad-lower-bound} to bound
\begin{equation*}
\int_{\mathbb{R}}\frac{1}{|h|^{1+2s}}\left\|\frac{b(k+h)-b(k)}{a(k+h)}-r_1(k)\frac{a(k+h)-a(k)}{a(k+h)}\right\|_{L^2_k}^2dh
\lesssim \frac{\|b(k)\|_{H^s_k}^2+\|a(k)\|_{H^s_k}^2}{c_0^2},
\end{equation*}
which completes the proof.
\end{proof}

\begin{proposition}\label{prop:q-r_j}
Let $s>1/2$.
Suppose $q\in H^{1,s}$ and $\|q\|_{L^1}\ll1$.
Then we have
\begin{equation}\label{eq:r-H^1}
r_1(k),r_2(k)\in H^s_k\cap L^{2,1}_k,
\end{equation}
that is, $r_1(k),~r_2(k)\in X^s_k$.
Moreover, the map
\begin{equation}\label{eq:r-lipschitz}
\begin{array}{ccc}
H^{1,s}                     &\longrightarrow& X_k^s\times X_k^s                    \\
\rotatebox{90}{$\in$}&               & \rotatebox{90}{$\in$} \\
 q                    & \longmapsto   & (r_1 ,r_2)
\end{array}
\end{equation}
is locally Lipschitz.
\end{proposition}

\begin{proof}
The first part, $r_1(k),r_2(k)\in H^s_k$ follows from Lemma \ref{lem:r-upper-bound} .
Combining \eqref{eq:kb(k)} with Lemma \ref{lem:ad-lower-bound} implies $r_1(k),r_2(k)\in L^{2,1}_k$, which completes the proof of \eqref{eq:r-H^1}.
The Lipschitz continuity in \eqref{eq:r-lipschitz} is easily verified, by  commuting their proof.
\end{proof}

\section{Riemann-Hilbert problem}\label{sec:rhproblem}

In this section, we consider the Riemann-Hilbert problem on the real line.
Due to the structure of \eqref{eq:r-t0},  letting the two-by-two matrix $v=v(x,k)$ be the jump matrix on the real line such that
\begin{equation}\label{eq:C-R}
m_+(x,k)=m_-(x,k)v(x,k)\quad \text{for $k\in\mathbb{R}$}, 
\end{equation}
where
\begin{equation*}
\begin{split}
m_+(x,z)  & =
\begin{bmatrix}
\frac{m_1^{(+)}(x,z)}{a(z)} & m_2^{(-)}(x,z)
\end{bmatrix}
\quad \text{for $z\in\mathbb{C}_+$},\\
m_-(x,z) & =
\begin{bmatrix}
m_1^{(-)}(x,z) & \frac{m_2^{(+)}(x,z)}{d(z)}
\end{bmatrix}
\quad \text{for $z\in\mathbb{C}_-$}.
\end{split}
\end{equation*}
By Cramer's rule, the continuity condition yields that for $k\in\mathbb{R}$
\begin{equation}\label{eq:v(k)}
\begin{split}
v(x,k) & =
\begin{bmatrix}
 \frac{1}{a(k)d(k)}\det\begin{bmatrix} m_1^{(+)}(x,k) & m_2^{(+)}(x,k) \end{bmatrix}  & \frac{1}{d(k)}\det\begin{bmatrix} m_2^{(-)}(x,k) & m_2^{(+)}(x,k) \end{bmatrix}\\
\frac{1}{a(k)}\det\begin{bmatrix} m_1^{(-)}(x,k) & m_1^{(+)}(x,k) \end{bmatrix} & \det\begin{bmatrix} m_1^{(-)}(x,k) & m_2^{(-)}(x,k)  \end{bmatrix} 
\end{bmatrix}\\
& =
\begin{bmatrix}
\frac{1}{a(k)d(k)}  & -\frac{c(k)}{d(k)}e^{2ikx}\\
\frac{b(k)}{a(k)}e^{-2ikx} & 1
\end{bmatrix}\\
& = \begin{bmatrix}
1-r_1(k)r_2(k)  & -r_2(k)e^{2ikx}\\
r_1(k)e^{-2ikx} & 1
\end{bmatrix}\\
&  =  e^{ ikx\ad(\sigma)}v(0,k),
\end{split}
\end{equation}
where
\begin{equation*}
v(0,k)=
\begin{bmatrix}
1-r_1(k)r_2(k)  & -r_2(k)\\
r_1(k) & 1
\end{bmatrix}. 
\end{equation*}

\begin{remark}\label{rem:v(k)-mimic}
In a later section, we use the upper and lower factorization as following, respectively: 
$$
v(x,k)=\left(I-w_-(x,k)\right)^{-1}\left(I+w_+(x,k)\right),
$$
where
\begin{equation*}
\begin{split}
w_-(x,k) & =
\begin{bmatrix}
0 & -r_2(k)e^{2ikx} \\
0 & 0
\end{bmatrix},\\
w_+(x,k) & =
\begin{bmatrix}
0 & 0 \\
r_1(k)e^{-2ikx} & 0
\end{bmatrix}.
\end{split}
\end{equation*}
\end{remark}

With this observation, we are now able to take the Rieman-Hilbert problem.
For given $(r_1,r_2)\in X^s_k\times X^s_k$, letting $v(x,k)$ be given by \eqref{eq:v(k)}, namely $v(x,k)= e^{ ikx\ad(\sigma)}v(0,k)$.
For each $x\in\mathbb{R}$, we seek that a square matrix $M(x,z)$ of order two, which solves the following Hilbert-Riemann problem $(\mathbb{R},v)_{L^2}$ with respect to $z$:
\begin{itemize}
\item[(i)]
$M(x,z)$ is analytic in $\mathbb{C}\backslash\mathbb{R}$.
\item [(ii)]
$M_+(x,k)=M_-(x,k)v(x,k)$ for $k\in\mathbb{R}$, where $M_{\pm}(x,k)=\lim_{\mathbb{C}_{\pm}\ni z\to k}M(x,z)$, respectively.
\item[(iii)]
$\lim_{\mathbb{C}_{\pm}\ni z\to\infty}M_{\pm}(x,z)=I$.
\end{itemize}
We need to place the third conditional (iii) on the solution $M(x,z)$ to prove uniqueness.  

Let us fix the notation that we will use the paper.
Via an associated deformed Riemann-Hilbert problem,  we denote by $\mathscr{C}_v$ associated operator
$$
\mathscr{C}_{v(x,\cdot)}^{\pm}[h](k)=\mathscr{C}^{\pm}[h(\cdot)(v(x,\cdot)-I)](k)
$$
for $k\in\mathbb{R}$.
For simplicity, we will abbreviate $\mathscr{C}_{v(x,\cdot)}^{\pm}[h](k)$ to $\mathscr{C}_{v}^{\pm}[h](k)$, by omitting the variable $x$.

\begin{lemma}\label{lem:bijective}
Suppose $r_1,r_2\in H^s_k$ for $s>1/2$ with $\|r_1\|_{L^{\infty}_k}<1$ and $\|r_2\|_{L^{\infty}_k}<1$.
Then for all square matrices $F(k)\in L^2_k$ of order two, there exists a unique solution $($matrix$)$ $h_{\pm}(k)\in L^2_k$ such that
$$
(I\pm \mathscr{C}_v^{\pm})[h_{\pm}](k)=F(k),
$$
respectively.
\end{lemma}

\begin{proof}
For the proof, we refer to \cite{deift} and \cite[Proof of Lemma 3.2]{zhao}.
By direct computation, one finds that
\begin{equation*}
\begin{split}
\mathscr{C}_v^{\pm}\mathscr{C}_{v^{-1}}^{\pm}[h] & = \mathscr{C}^{\pm}\left[\mathscr{C}_{v^{-1}}^{\pm}[h](v-I)\right]\\
 & = \mathscr{C}^{\pm}\left[\left(\mathscr{C}^{\pm}[h(v^{-1}-I)]\right)(v-I)\right]\\
 & = \mathscr{C}^{\pm}\left[\left(\mathscr{C}^{\mp}[h(v^{-1}-I)]\right)(v-I)\right]\pm \mathscr{C}^{\pm}\left[h(v^{-1}-I)(v-I)\right]\\
 & = \mathscr{C}^{\pm}\left[\left(\mathscr{C}^{\mp}[h(v^{-1}-I)]\right)(v-I)\right]\mp \mathscr{C}_{v^{-1}}^{\pm}[h]\mp \mathscr{C}_v^{\pm}[h],
\end{split}
\end{equation*}
which yields
$$
(I\pm \mathscr{C}_v^{\pm})(I\pm \mathscr{C}_{v^{-1}}^{\pm})[h]=\mathscr{C}^{\pm}\left[\left(\mathscr{C}^{\mp}[h(v^{-1}-I)]\right)(v-I)\right]+h.
$$
Noting that $r_1,r_2\in H^{s}_k\subset C_k(\mathbb{R})\cap L^{\infty}_k(\mathbb{R})$, a standard argument based on Ascoli's theorem implies that the operator $\mathscr{C}^{\pm}\left[\left(\mathscr{C}^{\mp}[\cdot(v^{-1}-I)]\right)(v-I)\right]$ is compact in $L_k^2$.
Then $I\pm \mathscr{C}_v^{\pm}$ is Fredholm.
In particular, if $v=I$, then we have $I\pm \mathscr{C}_v^{\pm}=I$.
By the continuity argument, we see
$$
\textrm{ind}\,(I\pm \mathscr{C}_v^{\pm})=\textrm{ind}\, I=\dim\ker I-\mathrm{codim}\, R(I)=0-0=0.
$$
By Fredholm's theorem of the alternative, it suffices to verify $\dim\ker(I\pm \mathscr{C}_v^{\pm})=0$ in $L^2_k$.
 
Suppose that $g_{\pm}\in L^2_k$ satisfy $(I\pm \mathscr{C}_v^{\pm})g_{\pm}=0$ in $L^2_k$, respectively. 

Consider $\mathscr{C}[g_{\pm}(v-I)](z)$ the extension of $\mathscr{C}^{\pm}[g_{\pm}(v-I)](k)$ off $k\in \mathbb{R}$ to $z\in \mathbb{C}_{\mp}$ and letting $P_{\pm}(z)=\mathscr{C}[g_{\pm}(v-I)](z)(\mathscr{C}[g_{\pm}(v-I)](z))^*$, where $A^*$ denotes the Hermitian matrix of $A$.
Then by assumption, we may write
$$
P_\pm(k)=\mathscr{C}_v^{\mp}[g_{\pm}](k)(\mathscr{C}_v^{\pm}[g_{\pm}](k))^*=\left(g_{\pm}(k)+ [g_{\pm}(v-I)](k)\right)\left( g_{\pm}(k)\right)^*=g_{\pm}(k)v(x,k)\left( g_{\pm}(k)\right)^*
$$
for $k\in \mathbb{R}$.
Let $\Gamma_+(R,\varepsilon)$ be a semi-circle of radius $R>0$ in upper plane lying entirely upper the real axis $\Im z\ge \varepsilon$ centered at the origin, while $\Gamma_-(R,\varepsilon)$ be a semi-circle of radius $R>0$ in lower plane lying entirely lower the real axis $\Im z\le -\varepsilon$ centered at the origin.
Since $P_{\pm}(z)$ is analytic in $\mathbb{C}_{\mp}$, we have 
$$
\oint_{\Gamma_{\mp}(R,\varepsilon)}P_{\pm}(z)\,dz=O_{2\times 2},
$$
where $O_{2\times 2}$ is a null matrix, that is, every single entry in the matrix is exactly zero.
Since $v-I\in H^s_k$ and $g_{\pm}\in L^2_k$, we see
$$
\lim_{R\to \infty}\int_{C(R)_{\pm}}P_{\pm}(z)\,dz=O_{2\times 2},
$$
where $C(R)_{\pm}=\{z\in\mathbb{C}\mid \pm\Im z\ge 0,~|z|=R\}$.
Then
\begin{equation*}
O_{2\times 2}  =\int_{\mathbb{R}}P_{\pm}(k)\,dk
= \int_{\mathbb{R}}g_{\pm}(k)\left( g_{\pm}(k)\right)^*\,dk
= \int_{\mathbb{R}}g_{\pm}(k)v(x,k)\left( g_{\pm}(k)\right)^*\,dk.
\end{equation*}
By \cite[Lemma 3.1]{zhao}, due to $\|r_1\|_{L^{\infty}_k}<1$ and $\|r_2\|_{L^{\infty}_k}<1$, the real part of $w^*v(x,k)w$ is positive for every nonzero complex vector $w\in \mathbb{C}^2$.
Then taking the trace on both sides of the above matrix equation, we conclude $g_{\pm}=0$.
\end{proof}

\begin{corollary}\label{cor:existence-RH}
Let $s>1/2$.
Supposed $r_1,r_2\in H^s_k$ with $\|r_1\|_{L^{\infty}_k}<1$ and $\|r_1\|_{L^{\infty}_k}<1$.
Then for all $x\in\mathbb{R}$, there exists a unique solution $M_{\pm}(x,k)-I\in L^2_k$ to the Riemann-Hilbert problem $(\mathbb{R},v)_{L^2}$.
\end{corollary}

\begin{proof}
It suffices to show the existence of $m_{\pm}$ satisfying the Hilbert-Riemann problem $(\mathbb{R},v)_2$, because the uniqueness follows from the auxiliary condition $\lim_{\mathbb{C}_{\pm}\ni z\to\infty}M_{\pm}(x,z)=I$.

Noting that $v(x,k)-I\in L^2_k$.
Applying Lemma \ref{lem:bijective} with $F(k)=\mathscr{C}^-[v(x,\cdot)-I](k)\in L^2_k$, we see that there exists a unique $M_-(x,k)$ such that $M_-(x,k)-I\in L^2_k$ and $(I-\mathscr{C}_v^-)[M_-(x,\cdot)-I](k)= \mathscr{C}^-[v(x,\cdot)-I](k)$.

Noting that $M_-(x,k)(v(x,k)-I)\in L^2_k$ for all $x$, then $\mathscr{C}^+[M_-(x,\cdot)(v(x,\cdot)-I)](k)\in L^2_k$.
Letting $M_+(x,k)=I+\mathscr{C}^+_v[M_-(x,\cdot)](k)$, then we have
$$
M_+(x,k)=I+M_-(x,k)\left(v(x,k)-I\right)+\mathscr{C}^-[M_-(x,\cdot)(v(x,\cdot)-I)](k)=M_-(x,k)v(x,k).
$$
Moreover, in a similar way to the proof of Lemma \ref{lem:bijective}, the extension of $M_{\pm}(x,k)$ off $k\in\mathbb{R}$ to $\mathbb{C}_{\pm}$ is analytic function in $\mathbb{C}_{\pm}$, respectively.
\end{proof}

\begin{remark}\label{rem:M_+}
In the proof of Lemma \ref{lem:bijective}, the solution $M_-(x,k)$ was construct in the first step.
If one sets to work on finding $M_+(x,k)$ such that $(I+\mathscr{C}_{v^{-1}}^+)[M_+(x,\cdot)-I](k)=-\mathscr{C}^+[v(x,\cdot)^{-1}-I](k)$, then $M_-(x,k)=I-\mathscr{C}_{v^{-1}}^-[M_+(x,\cdot)](k)$, which satisfied
$$
M_-(x,k)=I+M_+(x,k)\left(v(x,k)^{-1}-I\right)-\mathscr{C}^{+}[M_+(x,\cdot)(v^{-1}(x,\cdot)-I)](k)=M_+(x,k)v(x,k)^{-1},
$$
from which follows that $M_+(x,k)=M_-(x,k)v(x,k)$.
\end{remark}

\begin{remark}
Lemma \ref{lem:bijective} shows that the operators $I\pm \mathscr{C}_v^{\pm}$ have a bounded inverse on $L^2_k$.
From the structural steps conjunction with \eqref{eq:v(k)}, we have that
\begin{equation}\label{eq:M_L^2-boound}
\|M_{\pm}(x,\cdot)-I\|_{L_x^{\infty}(\mathbb{R},L^2_k(\mathbb{R}))}\lesssim \|r_1\|_{L^2_k}+\|r_2\|_{L^2_k}.
\end{equation}
\end{remark}

By indicating entries of matrix as a subscript, we write
\begin{equation}\label{eq:MM}
M_{\pm}(x,k)=
\begin{bmatrix}
M_1^{(\pm)}(x,k) & M_2^{(\pm)}(x,k)
\end{bmatrix}
\end{equation}
and
$$
M_1^{(\pm)}(x,k)=
\begin{bmatrix}
M_{1,1}^{(\pm)}(x,k) \\
M_{1,2}^{(\pm)}(x,k)
\end{bmatrix},
\quad
M_2^{(\pm)}(x,k)=
\begin{bmatrix}
M_{2,1}^{(\pm)}(x,k) \\
M_{2,2}^{(\pm)}(x,k)
\end{bmatrix},
$$
respectively.

Setting
\begin{equation*}
M(x,k)=
\begin{bmatrix}
M_{1}^{(-)}(x,k)  & M_{2}^{(+)}(x,k) 
\end{bmatrix}.
\end{equation*}

\begin{lemma}\label{lem:M_1^+M_2^-}
Let $s>1/2$.
Supposed $r_1,r_2\in H^s_k$ with $\|r_1\|_{L^{\infty}_k}<1$ and $\|r_1\|_{L^{\infty}_k}<1$.
Then the solutions $M_{\pm}(x,k)$ obtained in Corollary $\ref{cor:existence-RH}$ satisfy
\begin{equation}\label{eq:M_H^s}
\begin{bmatrix}
M_{1}^{(-)}(x,k) & M_{2}^{(+)}(x,k)
\end{bmatrix} 
-I
\in L_x^{\infty,s}(\mathbb{R}_-,L^2_k(\mathbb{R})).
\end{equation}
\end{lemma}

\begin{proof}
We mimic the proof in \cite[Lemma 3.4]{deift} and \cite[Lemma 3.5]{zhao}.
In virtue of \eqref{eq:C-R}, we have
\begin{equation}\label{eq:M_+-M_-}
M_-(x,k)\left(v(x,k)-I\right) =M_+(x,k)\left(I-v(x,k)^{-1}\right).
\end{equation}
Using the composition of $M_{-}(x,k)$ in the proof of Corollary \ref{cor:existence-RH} and the additional explanation for $M_+(x,k)$ in Remark \ref{rem:M_+}, we have
\begin{equation*}
\begin{split}
& M(x,k)-I-\mathscr{C}^-[(M(x,\cdot)-I)w_+(x,\cdot)](k)-\mathscr{C}^+[(M(x,\cdot)-I)w_-(x,\cdot)](k)\\
=& 
\begin{bmatrix}
0 & -\mathscr{C}^+[r_2(\cdot)e^{2ix\cdot}](k) \\
\mathscr{C}^-[r_1(\cdot)e^{-2ix\cdot}](k) & 0
\end{bmatrix}
\end{split}
\end{equation*}
It is worth noting that by using the same argument in Remark \ref{rem:v(k)-mimic}, we have
$$
w_+(x,k)+w_-(x,k)=\left(I-w_-(x,k)\right)\left(v(x,k)-I\right).
$$
Then
\begin{equation}\label{eq:M_pm}
\begin{split}
\left(I-\mathscr{C}^-_v\right)\left[(M(x,\cdot)-I)\left(I-w_-(x,\cdot)\right)\right](k) & =
\left(I+\mathscr{C}^+_{v^{-1}}\right)\left[(M(x,\cdot)-I)\left(I+w_+(x,\cdot)\right)\right](k) \\
& = \begin{bmatrix}
0 & -\mathscr{C}^+[r_2(\cdot)e^{2ix\cdot}](k)\\
\mathscr{C}^-[r_1(\cdot)e^{-2ix\cdot}](k) & 0
\end{bmatrix}
\end{split}
\end{equation}
From the expression
\begin{equation*}
\begin{split}
& \mathscr{C}^-[r_1(\cdot)e^{-2ix\cdot}](k)=\frac{1}{\sqrt{2\pi}}\int_{2x}^{\infty}e^{ik(-2x+\eta)}\mathscr{F}_k[r_1](\eta)\,dt=\frac{1}{\sqrt{2\pi}}\int_{-\infty}^{0}e^{ik\eta}\mathscr{F}_k[r_1](\eta+2x)\,d\eta,\\
& \mathscr{C}^+[r_2(\cdot)e^{2ix\cdot}](k)=\frac{1}{\sqrt{2\pi}}\int_{-\infty}^{-2x}e^{ik(2x+\eta)}\mathscr{F}_k[r_2](\eta)\,d\eta=\frac{1}{\sqrt{2\pi}}\int_{0}^{\infty}e^{ik\eta}\mathscr{F}_k[r_2](\eta-2x)\,d\eta,
\end{split}
\end{equation*}
it follows that for $x\in\mathbb{R}_-$
\begin{equation*}
\begin{split}
& \|\mathscr{C}^-[r_1(\cdot)e^{-2ix\cdot}](\cdot)\|_{L^2_k}\lesssim \| \mathscr{F}_k[r_1](\eta+2x)1_{\mathbb{R}_-}(\eta)\|_{L^2_{\eta}}\lesssim \langle x\rangle^{-s}\|r_1\|_{H^s_k},\\
& \|\mathscr{C}^+[r_2(\cdot)e^{2ix\cdot}](\cdot)\|_{L^2_k}\lesssim \| \mathscr{F}_k[r_2](\eta-2x)1_{\mathbb{R}_+}(\eta)\|_{L^2_{\eta}}\lesssim \langle x\rangle^{-s}\|r_2\|_{H^s_k}.
\end{split}
\end{equation*}
As stated already, the operator $(I-\mathscr{C}_v^-)^{-1}$ is bounded in $L^2_k$ for all $x\in\mathbb{R}$ and $(I-w_-(x,k))^{-1}$ does so.
Consequently, we have
$$
\|M(x,\cdot)-I\|_{L^2_k}\lesssim \langle x\rangle^{-s}\left(  \|r_1\|_{H^s_k}+\|r_2\|_{H^s_k}\right)
$$
for $x\le 0$, which completes the proof of \eqref{eq:M_H^s}.
\end{proof}

It can be true that for $M_{\pm}(x,k)$ obtained in Corollary \ref{cor:existence-RH} and Remark \ref{rem:M_+},
\begin{equation*}
\partial_xM_{\pm}(x,k)\in C_x(\mathbb{R},L_k^{2}(\mathbb{R}))
\end{equation*}
this is because $\partial_x(e^{\pm2ikx}r_1(k))\in L^2_k$ implies $M_{\pm}(x,k)-I=\mathcal{C}^{\pm}[h(x,\cdot)](k)$ for some $h(x,k)\in C_x^1(\mathbb{R},L^{2}_k)$.

\begin{lemma}\label{lem:M-derivative}
Let $s>1/2$.
Supposed $r_1,r_2\in H^s_k$ with $\|r_1\|_{L^{\infty}_k}<1$ and $\|r_1\|_{L^{\infty}_k}<1$.
In addition, if $r_1,r_2\in X^s$, then
\begin{equation}\label{eq:H_dx}
\partial_xM(x,k)-ik\ad\sigma_3\left(M(x,k)\right)\in L_x^{\infty}(\mathbb{R},L^2_k(\mathbb{R})).
\end{equation}
\end{lemma}

\begin{proof}
Noting that
$$
\partial_xv(x,k)=\partial_x(w_+(x,k)+w_-(x,k))=-2ik(w_+(x,k)-w_-(x,k))=ik\left[\sigma_3,v(x,k)\right]=ik\ad\sigma_3(v(t,x)),
$$
and
\begin{equation*}
\begin{split}
&\ad\sigma_3(M_-(x,k))v(x,k)-\ad\sigma_3(M_+(x,k)) \\
& = \sigma_3M_+(x,k)-M_-(x,k)\sigma_3v(x,k)-\sigma_3M_+(x,k)+M_-(x,k)v(x,k)\sigma_3\\
& = -M_-(x,k)\ad\sigma_3(v(x,k)),
\end{split}
\end{equation*}
that is,
$$
ik\ad\sigma_3(M_+(x,k))-ik\ad\sigma_3(M_-(x,k))v(x,k)=M_-(x,k)\partial_xv(x,k).
$$
Differentiating $M_+(x,k)=M_-(x,k)v(x,k)$ with respect to $x$, we then have
\begin{equation*}
\partial_xM_+(x,k)-ik\ad\sigma_3\left(M_+(x,k)\right)=\left(\partial_xM_-(x,k)-ik\ad\sigma_3\left(M_-(x,k)\right)\right)v(x,k).
\end{equation*}
This should be compared to \eqref{eq:volttera}, nothing the uniqueness of Riemann-Hilbert problem $(\mathbb{R},v)_{L^2}$ and Lemma \ref{lem:m^+-}.
Consequently, we have
\begin{equation*}
\partial_xM(x,k)-ik\ad\sigma_3\left(M(x,k)\right)=Q(q)M(x,k).
\end{equation*}
By \eqref{eq:M_L^2-boound}, the right-hand side term can be controlled in $L_x^{\infty}(\mathbb{R},L^2_k(\mathbb{R}))$, which completes the proof of \eqref{eq:H_dx}.
\end{proof}

Collecting earlier results as arguments above, we obtain the following proposition.

\begin{proposition}\label{prop:M_1^+M_2^-}
Let $1/2$.
Suppose $r_1,r_2\in X^s_k$ with $\|r_1\|_{L^{\infty}_k}<1$ and $\|r_2\|_{L^{\infty}_k}<1$.
Then the maps
\begin{equation*}
\begin{array}{ccc}
X_k^s\times X_k^s                      &\longrightarrow& L_x^{\infty,s}(\mathbb{R}_-,L^2_k(\mathbb{R}))                   \\
\rotatebox{90}{$\in$}&               & \rotatebox{90}{$\in$} \\
 (r_1 ,r_2)                    & \longmapsto   & M(x,k)-I
\end{array}
\end{equation*}
and
\begin{equation*}
\begin{array}{ccc}
X_k^s\times X_k^s                      &\longrightarrow&  L_x^{\infty}(\mathbb{R},L^2_k(\mathbb{R}))                  \\
\rotatebox{90}{$\in$}&               & \rotatebox{90}{$\in$} \\
 (r_1 ,r_2)                    & \longmapsto   & \partial_xM(x,k)-ik\ad\sigma_3\left(M(x,k)\right)
\end{array}
\end{equation*}
are locally Lipschitz.
\end{proposition}

\section{Reconstruction of the potential energy}\label{sec:rp}

Following \cite{deift,zhao}, we use the scattering data to construct the potential $q$ via locally Lipschitz maps $(r_1,r_2) \mapsto (q,r)$.
Suppose that $q\in H^{1,s}$ for $s>1/2$.
With \eqref{eq:m-limit2}, \eqref{eq:C-R} and \eqref{eq:MM} out of the way, $q$ and $r$ are reconstructed as follows; for all $x\in\mathbb{R}$
\begin{equation*}
\begin{split}
q(x)& =-\lim_{\Im k\to \infty}2ikm_{2,1}^{(-)}(x,k)\\
& =- \lim_{\Im k\to \infty}2ikM_{2,1}^{(+)}(x,k)\\
& =\lim_{\Im k\to \infty}2ik\left[\mathscr{C}^+_{v^{-1}}\left[M_{+}(x,\cdot)\right](k)\right]_{2,1}\\
& =-\lim_{\Im k\to \infty}2ik \mathscr{C}^+\left[M_{1,1}^{(-)}(x,\cdot)r_2(\cdot)e^{2ix\cdot}\right](k)
\end{split}
\end{equation*}
and
\begin{equation*}
\begin{split}
r(x)& =\lim_{\Im k\to \infty}2ik m_{1,2}^{(-)}(x,k)\\
& =-\lim_{\Im k\to \infty}2ik\left[\mathscr{C}^-_{v}\left[M_{-}(x,\cdot)\right](k)\right]_{1,2}\\
& =-\lim_{\Im k\to \infty}2ik\mathscr{C}^-\left[M_{2,2}^-(x,\cdot)r_1(\cdot)e^{-2ix\cdot}\right](k).
\end{split}
\end{equation*}
By \eqref{eq:M_+-M_-}, we have the following formulas (c.f. \cite[section 4]{deift} and \cite[section 4]{zhao})
\begin{equation}\label{eq:qr-ex}
\begin{split}
q(x)& =\frac{1}{\pi}\int_{\mathbb{R}}M_{1,1}^{(-)}(x,k)r_2(k)e^{2ikx}\,dk,\\
r(x)& =\frac{1}{\pi}\int_{\mathbb{R}}M_{2,2}^{(+)}(x,k)r_1(k)e^{-2ikx}\,dk.
\end{split}
\end{equation}
If $(r_1,r_2)=(r_1[q],r_2[q])$ are given in Proposition \ref{prop:q-r_j}, the above series of reconstruction procedure implies $r(x)=\sigma \overline{q(-x)}$.

Collecting the above lemmas, we have the following proposition.

\begin{proposition}\label{prop:red-pr}
Let $s>1/2$.
Suppose $r_1,r_2\in X^s_k$ with $\|r_1\|_{L_k^{\infty}}\ll 1$ and $\|r_2\|_{L_k^{\infty}}\ll 1$.
Then we have $q,r\in H^{1,s}$.
Moreover, the induced map by the relations in \eqref{eq:qr-ex}$;$
$$
\begin{array}{ccc}
X^s\times X^s                    &\longrightarrow& H^{1,s}(\mathbb{R}_-)\times  H^{1,s}(\mathbb{R}_-)                     \\
\rotatebox{90}{$\in$}&               & \rotatebox{90}{$\in$} \\
 (r_1,r_2)                    & \longmapsto   & (q,r)
\end{array}
$$
is locally Lipschitz.
\end{proposition}

\begin{proof}
For $q$, by Lemma \ref{lem:M_1^+M_2^-} and Proposition \ref{prop:M_1^+M_2^-}, we may write
$$
q(x) =\frac{1}{\pi}\int_{\mathbb{R}}\left(M_{1,1}^{(-)}(x,k)-1\right) r_2(k)e^{2ikx}\,dk+\frac{1}{\pi}\int_{\mathbb{R}}r_2(k)e^{2ikx}\,dk,
$$
We begin by discussing $q\in L^{2,s}$. 
To estimate the first term, by \eqref{eq:M_pm} we have
\begin{equation}\label{eq:M11-12}
M_{1,1}^{(-)}(x,k)-1 =\mathscr{C}^-\left[M_{2,1}^{(+)}(x,\cdot)r_1(\cdot)e^{-2ix\cdot}\right](k).
\end{equation}
By using \eqref{eq:C-iden} and \eqref{eq:C-orth}, elementary considerations yield that for $x\in\mathbb{R}_-$;
\begin{equation*}
\begin{split}
|q(x)| & \lesssim  \left|\int_{\mathbb{R}} \mathscr{C}^-\left[M_{2,1}^{(+)}(x,\cdot)r_1(\cdot)e^{-2ix\cdot}\right](k)r_2(k)e^{2ikx}\,dk\right|+\left|\mathscr{F}[r_2](-2x)\right|\\
& = \left| \int_{\mathbb{R}} M_{2,1}^{(+)}(x,k)r_1(k)e^{-2ixk}\mathscr{C}^+\left[r_2(\cdot)e^{2ix\cdot}\right](k)\,dk\right|+\left|\mathscr{F}[r_2](-2x)\right|\\
& \lesssim \langle x\rangle^{-2s} \|M_{2,1}^{(+)}(x,k)\|_{L_x^{\infty,s}(\mathbb{R}_-,L_k^2(\mathbb{R}))}\|r_2\|_{H_k^s}+|\mathscr{F}[r_2](-2x)|.
\end{split}
\end{equation*}
Then for $x\in\mathbb{R}_-$;
$$
\|q\|_{L^{2,s}(\mathbb{R}_-)}\lesssim\left( \|M_{2,1}^{(+)}(x,k)\|_{L_x^{\infty,s}(\mathbb{R}_-,L^2_k(\mathbb{R}))}+1\right)\|r_2(k)\|_{H^s_k},
$$
which is acceptable by Lemma \ref{lem:M_1^+M_2^-}.

Turning to $\|\partial_x q\|_{L^{2}(\mathbb{R}_-)}$, we obtain
\begin{equation}\label{eq:Dq}
\|\partial_xq\|_{L^{2}(\mathbb{R}_-) }  \lesssim \left\|\partial_x\int_{\mathbb{R}} \left(M_{1,1}^{(-)}(x,k)-1\right)r_2(k)e^{2ikx}\,dk\right\|_{L^{2}_x(\mathbb{R}_-)} +\|r_2\|_{L^{2,1}_k}.
\end{equation}

By using \eqref{eq:C-iden} and \eqref{eq:C-orth} along with \eqref{eq:M11-12}, a quick computation shows that
\begin{equation}\label{eq:H_dx1}
\begin{split}
& \partial_x\int_{\mathbb{R}} \left(M_{1,1}^{(-)}(x,k)-1\right)r_2(k)e^{2ikx}\,dk \\
 =& \int_{\mathbb{R}} \partial_xM_{1,1}^{(-)}(x,k)r_2(k)e^{2ikx}\,dk+2i\int_{\mathbb{R}} \left(M_{1,1}^{(-)}(x,k)-1\right)kr_2(k)e^{2ikx}\,dk  \\
= &-\int_{\mathbb{R}} \mathscr{C}^+\left[r_2(\cdot)e^{2ix\cdot}\right](\zeta) \left(\partial_xM_{2,1}^{(+)}(x,\zeta)-2i\zeta M_{2,1}^{(+)}(x,\zeta)\right)r_1(\zeta)e^{-2ix\zeta}\,d\zeta\\
& +2i\int_{\mathbb{R}} \left(M_{1,1}^{(-)}(x,k)-1\right)kr_2(k)e^{2ikx}\,dk.
\end{split}
\end{equation}
For the first term on the right-hand side of \eqref{eq:H_dx1}, we employ \eqref{eq:H_dx} to obtain
$$
\partial_xM_{2,1}^{(+)}(x,\zeta)-2i\zeta M_{2,1}^{(+)}(x,\zeta)\in L_x^{\infty}(\mathbb{R},L_{\zeta}^2(\mathbb{R})).
$$
Then
\begin{equation*}
\begin{split}
& \left|\int_{\mathbb{R}} \mathscr{C}^+\left[r_2(\cdot)e^{2ix\cdot}\right](\zeta) \left(\partial_xM_{2,1}^{(+)}(x,\zeta)-2ikM_{2,1}(x,k)\right)r_1(\zeta)e^{-2ix\zeta}\,d\zeta\right|\\
\lesssim & \|\mathscr{C}^+\left[r_2(\cdot)e^{2ix\cdot}\right](\zeta)\|_{L^2_{\zeta}}\|\partial_xM_{2,1}^{(+)}(x,\zeta)-2i\zeta M_{2,1}^{(+)}(x,\zeta)\|_{L_x^{\infty}(\mathbb{R},L_{\zeta}^2)}\|r_1(\zeta)\|_{L^{\infty}_{\zeta}}\\
\lesssim  & \langle x\rangle^{-s}C\left(\|r_1\|_{H^s_k},\|r_2\|_{H^s_k}\right)\in L^2_x,
\end{split}
\end{equation*}
which yields that the contribution of the first term on the right-hand side of \eqref{eq:H_dx1} to \eqref{eq:Dq} has a desired bound.
Alternatively, the second term on the on the right-hand side of \eqref{eq:H_dx1} is bounded by
$$
\langle x\rangle^{-s}\left\|M_{1,1}^{(-)}(x,k)-1\right\|_{L_x^{\infty,s}(\mathbb{R}_-,L^2_k(\mathbb{R}))}\|r_2\|_{L^{2,1}_k}\in L^2_x,
$$
which is acceptable as before.

The proof for  $\|\partial_x r\|_{H^{1,s}(\mathbb{R}_-)}$ is like above.
This completes the proof of the proposition.
\end{proof}

\section{Proof of Theorem \ref{thm:main}}\label{sec:main}

In this section, we give the proof of Theorem \ref{thm:main} by obtaining an a priori estimate.
The local well-posedness in $H^s$ for $s\ge 0$ holds in a similar way to that of \eqref{eq:nls} in \cite{cazenave,tsutsumi} with some minor changes.

Let us use time-dependent notation, such as $m^{(\pm)}(t,x,k)$ for $m^{(\pm)}(x,k)$ used it so far, if there is no risk of confusion.
We seek the time evolution of solution $(q(t),r(t))$ by its reflection coefficients $(r_1(t,x),r_2(t,x))$ via Lax pair representation $\partial_tL=[P,L]$.
Following \cite{deift,zhao}, if $q(t)=q(t,x)\in C_t((-T,T), H_x^{1,s}(\mathbb{R}))$ solves the equation of \eqref{eq:nnls} for some $T>0$, then plugging $L(q(t))m^{(\pm)}(t,x,k)=O_{2\times 2}$ into $\partial_tL=[P,L]$ with $\kappa=ik$ implies
$$
L(q(t))\left(\partial_tm^{(\pm)}(t,x,k)-P(q(t))m^{(\pm)}(t,x,k)\right)=O_{2\times 2}.
$$
Then the time evolution of $m^{(\pm)}(t,x,k)$ is defined as follows:
\begin{equation*}
\partial_t m^{(\pm)}(t,x,k)
=i\begin{bmatrix}
-2k^2 & 0 \\
0  &2k^2
\end{bmatrix}
m^{(\pm)}(t,x,k)
=-2ik^2\sigma_3m^{(\pm)}(t,x,k),
\end{equation*}
so that $m^{(\pm)}(t,x,k)=e^{-2ik^2 t\sigma_3}m^{(\pm)}(x,k)$, putting it another way;
\begin{equation*}
\begin{bmatrix}
m_1^{(\pm)}(t,x,k) & m_2^{(\pm)}(t,x,k)
\end{bmatrix} 
=
\begin{bmatrix}
e^{-2ik^2 t}m_1^{(\pm)}(x,k) & e^{2ik^2 t}m_2^{(\pm)}(x,k)
\end{bmatrix}.
\end{equation*}
These formulas show that the time evolution of
\begin{equation*}\label{eq;a-t}
a(t,k),~b(t,k),~c(t,k),~d(t,k),
\end{equation*}
\begin{equation*}\label{eq:r-t}
r_1(t,k),~r_2(t,k),
\end{equation*}
\begin{equation*}\label{eq:M-t}
M_{\pm}(t,x,k)=
\begin{bmatrix}
M_1^{(\pm)}(t,x,k) & M_2^{(\pm)}(t,x,k)
\end{bmatrix},
\end{equation*}
and
\begin{equation*}
\begin{split}
q(t,x)& =\frac{1}{\pi}\int_{\mathbb{R}}M_{1,1}^{(-)}(t,x,k)r_2(t,k)e^{2ikx}\,dk,\\
r(t,x)& =\frac{1}{\pi}\int_{\mathbb{R}}M_{2,2}^{(+)}(t,x,k)r_1(t,k)e^{-2ikx}\,dk.
\end{split}
\end{equation*}
In particular,
\begin{equation*}
\begin{split}
a(t,k)& =a(k),\\
b(t,k) & = e^{-4ik^2 t}b(k),\\
c(t,k) & = e^{4ik^2 t}c(k), \\
d(t,k) & =d(k),
\end{split}
\end{equation*}
and then
\begin{equation*}
r_1(t,k)=e^{-4ik^2 t}r_1(k),\quad
r_2(t,k)=e^{4ik^2 t}r_2(k).
\end{equation*}

\begin{lemma}\label{lem:r(t)-X}
Let $s>1/2$.
For all $t\in\mathbb{R}$, we have $r_1(t,k),r_2(t,k)\in X^s_k$ and have the following estimates$:$
\begin{equation*}
\begin{split}
\|r_1(t,k)\|_{X^s_k} & \lesssim \|r_1\|_{H^s_k}+\langle t \rangle^{s}\|r_1\|_{L^{2,s}_k},\\
\|r_2(t,k)\|_{X^s_k} & \lesssim \|r_2\|_{H^s_k}+\langle t \rangle^{s}\|r_2\|_{L^{2,s}_k}.
\end{split}
\end{equation*}
\end{lemma}

\begin{proof}
Since $|r_1(t,k)|=|r_1(k)|<1$ and $|r_2(t,k)|=|r_2(k)|$, it suffices to show that $r_1(t,k),r_2(t,k)\in H^s_k$.
The proof for $r_2(t,k)\in H^s_k$ follows from the same argument as that of $r_1(t,k)\in H^s_k$.
So, we only consider the term $r_1(t,k)$.

Noting $\|r_1(t,k)\|_{L^2_k}=\|r_1(k)\|_{L^2_k}$.
Following the definition, we have
\begin{equation*}
\begin{split}
& \int_{\mathbb{R}^2}\frac{|r_1(k+h) e^{-4i(k+h)^2 t}-r_1(k)e^{-4ik^2 t}|^2}{|h|^{1+2s}}\,dkdh\\
\lesssim  & \int_{\mathbb{R}^2}\frac{|r_1(k+h)-r_1(k)|^2}{|h|^{1+2s}}\,dkdh+ \int_{\mathbb{R}^2}\frac{|r_1(k)|^2| e^{-4i(2k+h)h t}-1|^2}{|h|^{1+2s}}\,dhdk\\
\lesssim & \|r_1(k)\|_{H^s_k}^2 +\langle t\rangle^{2s}\int_{\mathbb{R}}\langle k\rangle^{2s}|r_1(k)|^2\left(\int_{\mathbb{R}}\frac{ |e^{-4i\frac{(2k+h/\langle k\rangle\langle t\rangle)t}{\langle k\rangle\langle t\rangle}h}-1|^2}{|h|^{1+2s}}\,dh\right)\,dk\\
\lesssim & \|r_1(k)\|_{H^s_k}^2 +\langle t\rangle^{2s}\int_{\mathbb{R}}\langle k\rangle^{2s}|r_1(k)|^2\left(\int_{\mathbb{R}}\frac{\min\{1,|h|^2\}}{|h|^{1+2s}}\,dh\right)\,dk\\
\lesssim & \|r_1\|_{H^s_k}^2+\langle t \rangle^{2s}\|r_1\|_{L^{2,s}_k}^2,
\end{split}
\end{equation*}
which is easily seen to be acceptable. 
\end{proof}

The fact that for all $t\in\mathbb{R}$ 
$$
|r_1(t,k)|=|r_1(k)|<1,\quad |r_2(t,k)|=|r_2(k)|<1,\quad r_1(t,k)\in X^s_k,\quad r_2(t,k)\in X^s_k
$$
yields the success of the reconstruction procedure as in Proposition \ref{prop:red-pr}:
\begin{equation*}
\begin{split}
q(t,x)& =\frac{1}{\pi}\int_{\mathbb{R}}M_{1,1}^{(-)}(t,x,k)r_2(k)e^{2ikx+4ik^2t}\,dk,\\
r(t,x)& =\frac{1}{\pi}\int_{\mathbb{R}}M_{2,2}^{(+)}(t,x,k)r_1(k)e^{-2ikx-4ik^2 t}\,dk.
\end{split}
\end{equation*}
Here $M_{1,1}^{(-)}(t,x,k)$ and $M_{2,2}^{(+)}(t,x,k)$ obey the same estimates as $M_{1,1}^{(-)}(x,k)$ and $M_{2,2}^{(+)}(x,k)$ stated in Section \ref{sec:rhproblem}, respectively, but the estimates depends on $\|r_1(t,k)\|_{X^s_k}$ and $\|r_2(t,k)\|_{X^s_k}$.
This admits that the local solution $(q(t),r(t))\in C_t((-T,T),H^{1,s}(\mathbb{R}_-))$ can be continued global in time for any $T>0$, where finite-time blowup does not occur.
Therefore by Lemma \ref{lem:r(t)-X}, we have a global solution $(q(t),r(t))\in L_t(\mathbb{R},H^{1,s}(\mathbb{R}_-))\times L_t(\mathbb{R},H^{1,s}(\mathbb{R}_-))$ having at most polynomial growth of some power as $t\to\pm\infty$, that is; there exists $\alpha>0$ such that for all $t\in\mathbb{R}$
$$
\|q(t)\|_{H^{1,s}(\mathbb{R}_-)}+\|r(t)\|_{H^{1,s}(\mathbb{R}_-)}\lesssim \langle t\rangle^{\alpha}\|q(0)\|_{H^{1,s}}.
$$
Since $\|q(t)\|_{H^{1,s}(\mathbb{R}_+)}=\|r(t)\|_{H^{1,s}(\mathbb{R}_-)}$, we have
$$
\|q(t)\|_{H^{1,s}(\mathbb{R})}\le \langle t\rangle^{\alpha}C(\|q_0\|_{H^{1,s}}).
$$
which yields the proof of Theorem \ref{thm:main}.

\section*{Acknowledgement}

This work was supported by JSPS KAKENHI, Grant Number 25K07068.


\end{document}